\documentclass[11pt,a4paper]{amsart}

\usepackage{graphicx}

\usepackage{listings}

\usepackage[latin1]{inputenc}
\usepackage{amsmath, amsthm,  amsfonts, amscd}
\usepackage{amssymb} % need this to use \nmid
\usepackage{url}
\newtheorem{lemma}{Lemma}

\numberwithin{equation}{section} \numberwithin{lemma}{section}

\usepackage{cite}

\newcommand{\bb}{{\rm\bf{b}}}
\newcommand{\calN}{{\mathcal{N}}}
\newcommand{\kommentar}[1]{}

\newcommand{\RRe}{\rm{Re}}

\newcommand*{\reff}[1]{(\ref{#1})}
\newcommand{\tEN}{\widetilde{E}_N^{(a,b)}}
\newcommand{\EN}{E_N^{(\alpha,\beta)}}
\newcommand{\ENf}{\EN(\phi)}
\newcommand{\nuN}{\nu_N^{(\alpha,\beta)}}
\newcommand{\tCN}{\widetilde C_N^{(a,b)}}
\newcommand{\CN}{C_N^{(\alpha,\beta)}}
\newcommand{\JN}{J_N^{(a,b)}}

\newtheorem{remark}[lemma]{Remark}
\newtheorem{proposition}[lemma]{Proposition}

\begin{document}

\title[The lowest eigenvalue in Jacobi ensembles and Painlev\'e VI]{THE LOWEST EIGENVALUE OF JACOBI RANDOM MATRIX ENSEMBLES AND PAINLEV\'E VI}
%\vspace {2 in}

\author[Due\~nez]{Eduardo Due\~nez}\email{eduenez@math.utsa.edu}
\address{Department of Mathematics, University of Texas at San Antonio, San Antonio, TX 78249, USA}

\author[Huynh]{Duc Khiem Huynh}\email{dkhuynhms@gmail.com}
\address{Department of Pure Mathematics, University of Waterloo, Waterloo, ON, N2L 3G1, Canada}

\author[Keating]{Jon P. Keating}\email{J.P.Keating@bristol.ac.uk}
\address{School of Mathematics, University of Bristol, Bristol BS8 1TW, UK}

\author[Miller]{Steven J. Miller}\email{Steven.J.Miller@williams.edu}
\address{Department of Mathematics and Statistics, Williams College, Williamstown, MA 01267, USA}

\author[Snaith]{Nina C. Snaith}\email{N.C.Snaith@bristol.ac.uk}
\address{School of Mathematics, University of Bristol, Bristol BS8 1TW, UK}

%\author{E. Due$\tilde{\rm n}$ez, D.\ K.\ Huynh, J.\ P.\ Keating, S.\ J. Miller, \\
%and N.\ C.\ Snaith}
%\\ School of Mathematics,\\ University of Bristol,\\ Bristol BS8 1TW, UK}



\thanks{The second-named author was partially supported by EPSRC, a
  CRM postdoctoral fellowship and NSF grant DMS-0757627. The
  fourth-named author was partially supported by NSF grants
  DMS-0855257 and DMS-097006. The final author was supported by
  funding from EPSRC.  The first, second and fourth authors thank the
  University of Bristol for its hospitality and support during the
  preparation of this manuscript.}

\date{\today}

\thispagestyle{empty} \vspace{.5cm}
\begin{abstract}
  We present two complementary methods, each applicable in a different range,
  to evaluate the distribution of the lowest eigenvalue of random matrices in a
  Jacobi ensemble.  The first method solves an associated Painlev\'e~VI
  nonlinear differential equation numerically, with suitable initial conditions
  that we determine.  The second method proceeds via constructing the
  power-series expansion of the Painlev\'e~VI function. Our results are
  applied in a forthcoming paper
%~\cite{DHKMS1}
  in which we model the
  distribution of the first zero above the central point of elliptic curve
  $L$-function families of finite conductor and of conjecturally orthogonal
  symmetry.
\end{abstract}

\maketitle
\tableofcontents

\section{Introduction}

We present techniques for calculating numerically the distribution of the lowest
eigenvalue (or synonymously, we say the `first eigenvalue') of random matrices
in Jacobi ensembles~$\JN$. We proceed as follows.  We introduce the Jacobi
ensemble $J_N=\JN$ of $N\times N$ random matrices. We relate the distribution of
the lowest eigenvalue of matrices in $J_N$ to the probability $\EN(0;I)$ that a
Jacobi ensemble has no levels in some interval~$I=[t,1]$ for $0\le t\le1$. We
use two complementary methods to evaluate $\EN$ relying on its interpretation as
the Okamoto $\tau$-function of a Painlev\'e~VI system, along with an auxiliary
Hamiltonian function $h(t)$ for which Forrester and Witte \cite{ForrWitte04}
have established explicit differential equations of Painlev\'e~VI type. Our
first method uses the Selberg-Aomoto integral to obtain explicit initial
conditions for the Painlev\'e~IV equation satisfied by $h(t)$ which are valid
close to the edge $t = 1$.  With these in hand we provide the MATLAB code
(relying on its built-in ordinary differential equation solver) to numerically
evaluate $h(t)$ alongside~$\EN$.  The second method uses the Painlev\'e~VI
equation with explicit initial conditions at the edge $t = 0$ together with
power series manipulations to recursively find the power series expansions of
$h(t)$ and~$\EN$.  We implement this algorithm on SAGE using its ability to
perform power series manipulations and symbolic algebra.

The use of these two complementary methods is essential in order to compute
$h(t)$ and $\EN(t)$ accurately over the whole range $0\le t\le1$. The
Painlev\'e~VI equation and its solutions have singularities at the edges $t=0$ and
$t=1$. The first method uses a solution found starting from an explicit initial
condition at a point $t_0=1-\varepsilon$ close to~$1$, where $\varepsilon>0$ is
a small positive parameter we determine empirically. Such an explicit initial
condition is found in Sections~\ref{InitialConditionsJacobi}
and~\ref{Matlabcode}; however, the initial condition is 
correct only up to terms of size $O(\varepsilon^2)$. The errors introduced by such
approximation and by the numerical Runge-Kutta method result in a computed
solution whose range of reliability may not extend to $t$
close to the singularity at $t=0$. The second method, described in
Section~\ref{sec:power-series}, constructs a truncated but otherwise exact power
series for $h(t)$ about $t=0$ (the singularity at $t=0$ is handled indirectly)
up to terms of order $O(t^{D+1})$ where $D$ is the degree of the
truncation. Such a solution is reliable over any interval $[0,u]$ with $u<1$,
provided $D$ is large enough, though not necessarily over the entire interval
$[0,1]$ in view of the singularity at $t=1$. In Section~\ref{sec:numer-painl-solv} we analyze the range of parameters $a,b,N$
for which both methods are stable in the sense that both numerically computed
solutions agree in some subinterval $[u,v]$ of $(0,1)$, which implies that the
numerical solver is robust for this range of parameters. 

It is important to note that the methods we use apply to non-integer values
of~$N$.

Painlev\'e differential equations have played a role in many
problems in random matrix theory, ranging from the distribution of
the eigenvalues in the bulk to the largest and smallest
eigenvalues, and have been extensively studied; for our purposes,
the most relevant are the investigations of solutions to
Painlev\'e~VI. We briefly mention some of the literature. We refer
the reader to the special edition of the Journal of Physics A
(Volume 39, Number 39, 2006), which celebrates 100 years of
Painlev\'e~VI, especially the historical introduction and survey
\cite{CJMNN} and the article by Forrester and Witte \cite{FW06} on
connections with random matrix theory; see also the recent works
by Dai and Zhang \cite{DZ09} and Chen and Zhang \cite{CZ09} for
determinantal formulas obtained from ladder operators.

The main contribution of this paper is the derivation of an
algorithm to compute numerically the distribution of the lowest
eigenvalue in the Jacobi ensembles, and a discussion of its
implementation and accuracy.  The motivation for this project
comes from attempts to understand the observations in \cite{Mil06}
on the distribution of the first zero above the central point in
families of elliptic curve $L$-functions when the conductors are
small. The Katz-Sarnak conjectures \cite{KatzSarnak99a,
KatzSarnak99b} predict that as the conductors of the elliptic
curves tend to infinity their zero statistics should agree with
the $N\to\infty$ scaling limits of the corresponding statistics of
the eigenvalues of matrices from a classical compact group. For
suitable test functions this was proved in \cite{Mil04_a,
Young06}; however, for finite conductors the numerical data in
\cite{Mil06} is in sharp disagreement with the limiting behavior
of these random matrix ensembles. In particular, the first zero
above the central point is repelled, with the repulsion decreasing
as the conductors increase. In a forthcoming paper %\cite{DHKMS1}
we complete the study of the low lying zeros of elliptic curve
$L$-functions, and obtain a model which describes the behavior of
these zeros for finite conductors. One of the key ingredients in
our model is the lowest eigenvalue of these Jacobi ensembles of
$N\times N$ matrices, often requiring non-integer values of $N$,
which is the main result of this paper.

%(the other ingredient is to consider the sub-ensemble of matrices
%whose characteristic polynomial exceeds a given threshold, which
%corresponds to the discretization of values of elliptic curve
%$L$-functions at the central point).

%%%%%%%%%%%%%%%%%%%%%%%%%%%%%%%%%%%%%%%%%%%%%%%%%%%%%%%%%%%%%%%%%%%%%%%%%%%%%%%%%%%%%%%%%%%%%%%%%%%%%%%%%%%%%%%%
%%%%%%%%%%%%%%%%%%%%%%%%%%%%%%%%%%%%%%%%%%%%%%%%%%%%%%%%%%%%%%%%%%%%%%%%%%%%%%%%%%%%%%%%%%%%%%%%%%%%%%%%%%%%%%%%
%%%%%%%%%%%%%%%%%%%%%%%%%%%%%%%%%%%%%%%%%%%%%%%%%%%%%%%%%%%%%%%%%%%%%%%%%%%%%%%%%%%%%%%%%%%%%%%%%%%%%%%%%%%%%%%%

\section{Jacobi ensembles and their first eigenvalue}
\label{SectionJacobiEnsemble}
Let $J_N=J^{{(a,b)}}_N$ denote the Jacobi ensemble on $N$~levels $0\le
x_j\le 1$, $j=1,2,\dots, N$, with real parameters $a,b>-1$. Explicitly, the
$N$-level (joint) probability density of levels of $J_N$ on $[0,1]^N$ with
respect to its Lebesgue measure $dx_1\,dx_2\cdots dx_N$ is given by
{\allowdisplaybreaks
\begin{equation}\label{eq:Jacobi-cos-PDF}
\tCN  \prod_{j=1}^N W(x_j) \prod_{1 \leq j < k  \leq N} (x_k - x_j)^2
\end{equation}
}
where the weight function $W=W^{{(a,b)}}$ on $[0,1]$ is given by
\begin{equation}
W(x) = x^b (1-x)^a
\end{equation}
and $\tCN $ is the ensemble's normalization constant.  Jacobi ensembles as
described above correspond to suitable ensembles of self-dual random matrices via
the angular variables~$\phi_j$ defined by
\begin{equation}
x_j = \frac{1 + \cos \phi_j}{2},\qquad 0 \le \phi_{j} \le \pi.
\end{equation}
Note that the edges $x=0$, $x=1$ correspond respectively to $\phi=\pi$, $\phi=0$
under this change of variables.
We refer the reader to~\cite{Duenez2004} and the forthcoming
book~\cite{LGRM} for details regarding the matrix realizations of Jacobi
ensembles, for which we will otherwise have no direct use.
In what follows we will go back and forth between the absciss\ae\ $x_j$ and the
angular variables $\phi_j$, but will in any case refer to the associated
ensemble by the Jacobi name and denote it by~$J_N$.  In terms of the angular
variables, and with respect to Lebesgue measure on $(0,\pi)^{N}$, the $N$-level
(joint) probability density for $J_N$ is given by
\begin{equation} \label{JacobiPDF} \CN \prod_{j=1}^N w(\cos \phi_j)
 \prod_{1 \leq j < k  \leq N} (\cos \phi_k - \cos \phi_j)^2
\end{equation}
with the weight function $w=w^{{(\alpha,\beta)}}$ on $(0,\pi)$,
\begin{equation} \label{Jacobiweight}
w(\cos\phi) = (1 - \cos\phi)^{\alpha} (1 + \cos\phi)^\beta.
\end{equation}
The parameters $\alpha,\beta>-\frac12$ are related to $a,b$ above by
$\alpha=a+1/2$, $\beta=b+1/2$, and $\CN $ is the appropriate normalization
constant, namely $\CN$ $=$ $2^{-(N+a+b+2)N}$ $\widetilde
C_N^{(\alpha-1/2,\beta-1/2)}$ for the constant $\tCN $
of~\eqref{eq:Jacobi-cos-PDF}. For suitable choices for $\alpha$ and $\beta$ we
obtain the joint probability density of the $N$ independent eigenphases for the
classical groups of matrices SO$(2N)$, SO$(2N+1)$ and USp$(2N)$, when the latter
are endowed with an invariant (Haar) probability measure and regarded as random
matrix ensembles. The case $\alpha = \beta = 0$ corresponds to SO$(2N)$, $\alpha
= 1$ and $\beta = 0$ corresponds to SO$(2N + 1)$, and $\alpha = \beta = 1$ to
USp$(2N)$. This is explained in detail in~\cite{Duenez2004}. Below in
\reff{constant} we give an explicit expression for the normalization
constant~$\tCN$. (Jacobi-distributed pseudorandom sequences of levels can be
generated from a uniform pseudorandom sequence using only the Jacobi joint
probability density via for instance the Accept-Reject Algorithm whose
applicability is quite broad; see for instance~\cite{RC}.)

As remarked above, the Jacobi ensemble $J_N$ describes the eigenvalue
statistics in suitable ensembles of self-dual random matrices having $N$ pairs
of eigenvalues $e^{\pm i\phi_{j}}$, $j=1,2,\dots,N$; we call $\phi$ the
eigenphase of the eigenvalue $e^{i\phi}$.  Let $\EN(n; I)$ denote the
probability that a random matrix $A \in J_N$ has exactly $n$ eigenphases in the
interval $I = [0,\phi]$. As shorthand notation we write $\EN(\phi)$ for
$\EN(0;[0,\phi])$, namely the probability of having no eigenphases in the
interval $[0,\phi]$.  The probability density function $\nuN(\phi)$ of the
distribution of the first eigenphase is related to $\EN(\phi)$ by
\begin{equation} \label{distributionoffirsteigenphases}
\nuN(\phi) = -\frac{d}{d \phi} \EN(\phi).
\end{equation}
We can deduce the relation \reff{distributionoffirsteigenphases} as follows:
assume that the interval $[0,\phi]$ contains no eigenvalues. Then a small
increment $\varepsilon > 0$ of the interval to $[0, \phi + \varepsilon]$ has two
possible outcomes.  Either the interval $[0, \phi + \varepsilon]$ contains no
eigenvalues, or it contains some. The probability of the first event is $\EN(\phi+\varepsilon)$. It follows that
$\EN(\phi) - \EN(\phi+\varepsilon)$ is the probability that the
interval $[\phi, \phi+ \varepsilon]$ contains at least one eigenvalue; as $\varepsilon
\rightarrow 0$ there can be only one, namely the first eigenphase in $[\phi,
\phi+ \varepsilon]$.  Thus
\begin{equation}
\lim_{\varepsilon \rightarrow 0} \frac{\EN(\phi) - \EN(\phi + \varepsilon)}{\varepsilon} = -\frac{d}{d \phi} \EN(\phi)
\end{equation}
indeed yields the probability density function $\nuN(\phi)$ of the first
eigenphase. An alternative way to prove \eqref{distributionoffirsteigenphases}
is to observe that $1-\EN(\phi)$ is the probability that $[0,\phi]$ contains at
least one eigenphase, hence that the first eigenphase $\phi_{\min}$ is at most
$\phi$; otherwise said $1-\EN(\phi)$ is the cumulative distribution function of
the first eigenphase $\phi_{\min}$, so its derivative is equal to the
probability density function $\nuN(\phi)$. In Section~\ref{Matlabcode} we shall
need to scale the angular variable $\phi$ by a factor of $N/\pi$ in order to
consider eigenphases of mean unit spacing on $[0,N]$.

We have
\begin{equation} \label{EGN}
  \begin{split}
\EN(\phi) & = \CN  \int_\phi^\pi \cdots \int_\phi^\pi
\prod_{j = 1}^N (1 -\cos\phi_j)^\alpha (1 +\cos\phi_j)^\beta \\
& \quad \times \prod_{1 \leq j < k \leq N} (\cos\phi_j - \cos\phi_k)^2
\,d \phi_1 \cdots d \phi_N
\end{split}
\end{equation}
for fixed $\alpha,\beta > -1/2$ and the normalization constant $\CN$
of~\eqref{JacobiPDF}.  There is no known method to evaluate the multiple
integral in equation~\reff{EGN} exactly. $\ENf$ is related to a Painlev\' e~VI
transcendental function $h(t)$, namely a certain solution to a second-order
nonlinear ordinary differential equation. In
Proposition~\ref{InitialConditionProposition} we provide the first few terms of
a power-series expansion of $\ENf$ for $\phi$ close to 0; these provide the
initial conditions for the differential equation we aim to solve. Our main
reference for the theory is the work of Forrester and Witte \cite{ForrWitte04}.
Their result is stated for the abscissal counterpart to the function $\EN(\phi)$
of~(\ref{EGN}), namely the function $\tEN(t)$ defined by
\begin{equation} \label{Ewdefinition} \tEN(t) = \tCN \int_0^t \cdots \int_0^t
  \prod_{j = 1}^N x_j^b(1-x_j)^a \prod_{1 \leq j < k \leq N}(x_j - x_k)^2 d x_1
  \cdots d x_N
\end{equation}
with the normalization constant $\tCN $
of~\eqref{eq:Jacobi-cos-PDF}. The functions \reff{Ewdefinition}
and \reff{EGN} are related by the change of variables
\begin{equation}
a = \alpha -1/2 \quad\mbox{and}\quad b = \beta -1/2
\end{equation}
along with
\begin{equation}
\label{eq:t-vs-phi}
t = \frac{1 + \cos \phi}{2}\quad \mbox{and}\quad x_j = \frac{1 + \cos \phi_j}{2}
\end{equation}
where $0 \le t, x_j \le 1$. Explicitly,
\begin{equation}
\EN(\phi) = \tEN\left(\frac{1 + \cos \phi}2\right).
\end{equation}

%%%%%%%%%%%%%%%%%%%%%%%%%%%%%%%%%%%%%%%%%%%%%%%%%%%%%%%%%%%%%%%%%%%%%%%%%%%%%%%%%%%%%%%%%%%%%%%%%%%%%%%%%%%%%%%%
%%%%%%%%%%%%%%%%%%%%%%%%%%%%%%%%%%%%%%%%%%%%%%%%%%%%%%%%%%%%%%%%%%%%%%%%%%%%%%%%%%%%%%%%%%%%%%%%%%%%%%%%%%%%%%%%
%%%%%%%%%%%%%%%%%%%%%%%%%%%%%%%%%%%%%%%%%%%%%%%%%%%%%%%%%%%%%%%%%%%%%%%%%%%%%%%%%%%%%%%%%%%%%%%%%%%%%%%%%%%%%%%%

\section{First method: The auxiliary Hamiltonian and
  Painlev\'e~VI}\label{Method1-PainleveEq}

Both of our mutually complementary methods rely on the interpretation of $\tEN$
as an Okamoto $\tau$-function and ensuing relation to a Painlev\'e system with
associated auxiliary Hamiltonian $h(t)$; this Hamiltonian arises as the solution
of a Painlev\'e~VI equation with the exact parameters determined by Forrester
and Witte in Proposition 13 of~\cite{ForrWitte04} as follows.

\begin{proposition} \label{Painleveproposition}Let $a,b > -1$ and $N$ be a positive integer.
The auxiliary Hamiltonian
\begin{equation} \label{Okamoto}
h(t) = t \cdot e_2'[\bb] - \tfrac{1}{2}e_2[\bb] + t(t-1)\tfrac{d}{dt} \log \tEN(t)
\end{equation}
where
\begin{align} \label{definitionofbs}
\bb & = (b_1,b_2,b_3,b_4) = \bigg(N + \frac{a + b}{2}, \frac{a-b}{2}, - \frac{a+b}{2}, -N-\frac{a+b}{2} \bigg),\\ \label{e2}
e_2[\bb] & = b_1 b_2 + b_1 b_3+ b_1 b_4 + b_2 b_3 + b_2 b_4 + b_3 b_4,\\ \label{e2p}
e'_2[\bb] & = b_1 b_3 + b_1 b_4 + b_3 b_4,
\end{align}
satisfies the following Painlev\'{e}~VI equation in Jimbo-Miwa-Okamoto $\sigma$-form
\begin{equation} \label{Painleve6}
h'(t)\left(t(1-t)h''(t) \right)^2 + \left(h'(t)[2h - (2t-1)h'(t)] + b_1b_2b_3b_4\right)^2 = \prod_{k=1}^4(h'(t) +b_k^2).
\end{equation}
Furthermore, we have the boundary condition (as $t \rightarrow 0$)
\begin{equation}\label{t=0-BoundCond}
h(t) = \left(-\frac{1}{2}e_2[\bb] -N(b + N) \right) + \left(e_2'[\bb] + \frac{N(N+b)(2N + a + b)}{2N + b} \right)t + O(t^2).
\end{equation}
\end{proposition}
Note that besides simplifying the notation in Proposition 13 of
\cite{ForrWitte04} we also swap the $a$'s and $b$'s therein. The parameter $a$
is equal to the order of vanishing of the Jacobi level density at the edge
$t=1$, whereas $a+1/2$ is the order of vanishing of eigenphase density at the
edge $\phi=0$. We remark that the apostrophe in the symbol $e_2'$ has no
specific meaning and is merely used to visually distinguish it from~$e_2$ (in a
manner consistent with the notation of reference~\cite{ForrWitte04}), whereas
the apostrophe in $h'$ and elsewhere in this manuscript means differentiation:
$h'(t)=\frac{dh}{dt}$, $h''(t)=\frac{d^2h}{dt^2}$,
${\tEN}{}'(t)=\frac{d}{dt}\tEN(t)$, etc.

Pay close attention to the fact that the initial condition given
in~\eqref{t=0-BoundCond} holds at $t=0$. This condition will be used in
Section~\ref{sec:power-series} to construct a power-series solution. For our
intended application, however, we are most interested in the behaviour of
$\EN(\phi)$ for $\phi$ close to zero, which in view of the change of
variables~\eqref{eq:t-vs-phi} corresponds to $t$ close to~$1$. Unfortunately,
the singularity of the Painlev\'e equation at $t=1$ significantly complicates
the numerical evaluation of the function $h(t)$ in this range.

Our first method of solution will numerically compute $h(t)$ starting instead
from an initial condition given at some fixed point $t=t_0$ close to~$1$, say
$t_0=1-\varepsilon$ for some small positive~$\varepsilon$ to be chosen
empirically. The determination of this suitable initial condition is a delicate
issue that depends on the analysis carried out in
Section~\ref{InitialConditionsJacobi}.

Following Edelman and Persson~\cite{Edel06}, we seek to compute
simultaneously $\tEN(t)$ and the Hamiltonian $h(t)$ via a (non-autonomous)
differential equation for the triple of functions
\begin{equation} \label{vectorH}
H(t)
=\left(
\begin{array}{*{3}{c}}
\tEN(t) \\
h(t) \\
h'(t) \\
\end{array}
\right)
\end{equation}
of the form
\begin{equation} \label{Painlevesystem}
\frac{dH}{dt}
= F_t(H)
\end{equation}
with initial conditions
\begin{equation} \label{InitialConditionsMatLab}
H_0 = H(t_0)
=\left(
\begin{array}{*{3}{c}}
\tEN(t_0) \\
h(t_0) \\
h'(t_0)
\end{array}
\right)
\end{equation}
where $t_0 = 1 - \varepsilon$ for small $\varepsilon > 0$. (The singularity of
the Painlev\'e equation and its solutions at $t_0=1$ preclude taking simply
$\varepsilon=0$.)

From \reff{Okamoto} we obtain
\begin{equation} \label{dHfirstcomponent}
\frac{d}{d t} \tEN(t) = \frac{h(t) - t e_2'[\bb] + \tfrac{1}{2}e_2[\bb]}{t(t-1)} \tEN(t),
\end{equation}
and likewise from \reff{Painleve6}
\begin{equation} \label{dHthirdcomponent}
h''(t) = \frac{1}{t(1-t)} \sqrt{\frac{\prod_{j=1}^4 (h'(t)+b_j^2)-\left(h'(t)[2h-(2t-1)h'(t)] + b_1 b_2 b_3 b_4\right)^2}{h'(t)}}.
\end{equation}
Therefore, (\ref{Painlevesystem}) holds with
\begin{equation}
  \label{F(H)}
  F_t
  \begin{pmatrix}
    h_1(t) \\ \ \\ h_2(t) \\ \ \\ h_3(t)
  \end{pmatrix}
  =
  \begin{pmatrix}
    \dfrac{h_2(t) - t e_2'[\bb] + \tfrac{1}{2}e_2[\bb]}{t(t-1)} h_1(t)\\ \ \\
    h_3(t) \\ \ \\
    \dfrac{1}{t(1-t)} \sqrt{\dfrac{\prod_{j=1}^4
        (h_3(t)+b_j^2)-\left(h_3(t)[2h_2(t)-(2t-1)h_3(t)] + b_1 b_2 b_3 b_4\right)^2}{h_2(t)}}
 \end{pmatrix}.
\end{equation}
The MATLAB code to compute $F_t(H)$ is given in the appendix and
is also available for download at \url{http://www.maths.bris.ac.uk/~mancs/publications.html}.
We employ the built-in ordinary differential solver \texttt{ode45} from MATLAB, which
implements a Runge-Kutta method giving an approximate solution
of~\reff{Painlevesystem} (and thus of the sought density $\nuN(\phi) = -
\tfrac{d}{d\phi}\EN(\phi)$ of the distribution of the first eigenphase).
It remains still to determine the initial condition $H_0=H(t_0)$ for
\reff{vectorH} as follows.
We shall find the small-$\varepsilon$ asymptotic behavior of $\tEN(1-\varepsilon)$
(equivalently, what we will actually do is find the small-$\phi$ asymptotics of $\EN(\phi)$).
By differentiation we then find
\begin{equation}
\frac{d}{dt}\tEN(t)\quad
\mbox{and}\quad
\frac{d^2}{dt^2}\tEN(t)
\end{equation}
and thus (through its definition~\eqref{Okamoto}) we obtain asymptotically good
approximations to $h(t_0)$ and $h'(t_0)$ for any $t_0$ close to 1. This gives
the triple $H_0$ of initial conditions for~\reff{vectorH}. In the following
section we compute the asymptotic behavior of $\EN(\phi)$ for $\phi$ close to~0.

%%%%%%%%%%%%%%%%%%%%%%%%%%%%%%%%%%%%%%%%%%%%%%%%%%%%%%%%%%%%%%%%%%%%%%%%%%%%%%%%%%%%%%%%%%%%%%%%%%%%%%%%%%%%%%%%
%%%%%%%%%%%%%%%%%%%%%%%%%%%%%%%%%%%%%%%%%%%%%%%%%%%%%%%%%%%%%%%%%%%%%%%%%%%%%%%%%%%%%%%%%%%%%%%%%%%%%%%%%%%%%%%%
%%%%%%%%%%%%%%%%%%%%%%%%%%%%%%%%%%%%%%%%%%%%%%%%%%%%%%%%%%%%%%%%%%%%%%%%%%%%%%%%%%%%%%%%%%%%%%%%%%%%%%%%%%%%%%%%

\section{Taylor series expansion for $\EN(\phi)$}
  \label{InitialConditionsJacobi} In this section we compute the
probability $\ENf$ that a random matrix from a Jacobi ensemble $J_N$ as defined
in Section \ref{SectionJacobiEnsemble}
has no eigenphase in the interval $[0, \phi]$ for small $\phi >
0$. As described at the end of Section
\ref{Method1-PainleveEq}, this enables us to
derive the initial conditions for the system of differential
equations \reff{Painlevesystem} which gives the distribution of
the first eigenphase \reff{distributionoffirsteigenphases}.
Forrester and Witte \cite{ForrWitte04} consider this same limit in
their equation (1.38), but here we derive a further term in the
approximation.  Our result is stated in Proposition
\ref{InitialConditionProposition}.

We require some notation.
For $n = 1, \ldots, N$ and $\alpha,\beta > -1/2$ we define the integral
\begin{equation}\label{Ih}
  \begin{split}
    I(n) & \ :=\   \CN  \underbrace{\int_0^\pi \cdots
      \int_0^\pi}_{N - n \text{~times}} \underbrace{\int_0^\phi \cdots
      \int_0^\phi}_{n \text{~times}} \prod_{j = 1}^N (1 -\cos\phi_j)^\alpha (1
    +\cos\phi_j)^\beta\\
    & \quad \ \ \ \ \ \ \times\ \prod_{1 \leq j < k \leq N} (\cos\phi_j - \cos\phi_k)^2 d \phi_1
    \cdots d \phi_N.
  \end{split}
\end{equation}

Then we have the following lemmata:
\begin{lemma} \label{lemmaEGN} For $\EN(\phi)$ as given in \reff{EGN} and
  for $I(n)$ as defined in~\reff{Ih} we have
\begin{equation} \label{lemmaEGNequation}
\EN( \phi)\ =\ 1 - N \cdot I(1) + {N \choose 2} I(2) - {N \choose 3} I(3) + \cdots + (-1)^N I(N).
\end{equation}
\end{lemma}

\newcommand{\al}{\alpha}
\begin{lemma} \label{lemmaIone}
For $I(1)$ as defined in \reff{Ih} we have
\begin{equation} \label{Ione}
I(1)\ =\ H_1 \frac{\phi^{2\al+1}}{2\al+1} - \left[(N-1) H_2 + \bigg(\frac{\al}{12} + \frac{\beta}{4} \bigg) H_1 \right] \frac{\phi^{2\al +3}}{2\al+3} + O(\phi^{2\al+4})
\end{equation}
where
\begin{equation} \label{Hone}
H_1 \ := \frac{\Gamma(\al + N + 1/2) \Gamma(\al + \beta + N)}{2^{2\al}\Gamma(\al + 1/2) \Gamma(\al + 3/2) \Gamma(N + 1) \Gamma(\beta + N - 1/2)}
\end{equation}
and
\begin{equation} \label{Htwo}
H_2\ :=\ \frac{\Gamma(\al + N + 1/2) \Gamma(\al + \beta + N + 1)}{2^{2\al+1}\Gamma(N + 1) \Gamma(\beta + N - 1/2) \Gamma(\al + 1/2) \Gamma(\al + 5/2)}.
\end{equation}
\end{lemma}

\begin{lemma} \label{lemmaIH}
For $I(n)$ as defined in~\reff{Ih} we have for $n \geq 2$ and $\alpha > -1/2$
\begin{equation}
I(n)\ \ll\ \phi^{2\al n + 2n^2 -n}.
\end{equation}
\end{lemma}

We postpone the proof of the above lemmata for a moment in order to state the desired Taylor series expansion of $E_N^{(\alpha,\beta)}(\phi)$ for small $\phi>0$.

\begin{proposition} \label{InitialConditionProposition}
For $\EN(\phi)$ as given in \reff{EGN} we have
\begin{equation}
  \begin{split}
    \EN(\phi) &\ =\ 1 - N \bigg(H_1 \frac{\phi^{2\al+1}}{2\al+1} -
    \left[(N-1) H_2 + \bigg(\frac{\al}{12} + \frac{\beta}{4} \bigg) H_1 \right]
    \frac{\phi^{2\al +3}}{2\al+3}\bigg) \\
    & \quad \ \ \ +\ O(\phi^{2\al+4})
  \end{split}
\end{equation}
where $H_1$ is defined in~\reff{Hone} and $H_2$ in~\reff{Htwo}.
\end{proposition}

\begin{proof} Let $n \geq 2$. We can apply Lemma \ref{lemmaIH} to every term
  $I(n)$ in~\reff{lemmaEGNequation} of Lemma~\ref{lemmaEGN}. Thus, $I(n) \ll
  \phi^{2\al n + 2n^2 -n}$. As $n \ge 2$, we have $2\al n + 2n^2 - n \geq 2\al + 4$ for every
  $\alpha > -1/2$, hence $I(n) \ll \phi^{2\al + 4}$. Thus every $I(n)$ in Lemma
  \ref{lemmaEGN} can be absorbed into the error term $O(\phi^{2 \al + 4})$ of $I(1)$ in
  \reff{Ione} of Lemma \ref{lemmaIone}.
\end{proof}

Now we prove Lemmata~\ref{lemmaEGN}, \ref{lemmaIone} and~\ref{lemmaIH}.
\begin{proof}[Proof of Lemma \ref{lemmaEGN}]
Recall that $\EN(\phi)$ is the probability of having no eigenphases in the interval $[0,\phi]$.
We denote the event that there is at least one eigenphase in $[0,\phi]$ by $\mathcal{B}$.
Then its complement $\complement\mathcal{B}$ is the event of having no eigenphases in $[0, \phi]$.
Hence
\begin{equation}
\EN(\phi) = P(\complement{\mathcal{B}}) = 1 - P({\mathcal{B}}).
\end{equation}

We now focus on $P({\mathcal{B}})$.
Let
\begin{equation}
B_k := \{(\phi_1, \ldots, \phi_k, \ldots, \phi_N) \in [0,\pi]^N \text{~where~} \phi_k \in [0,\phi] \},
\end{equation}
which is the event that $\phi_k$ lies in $[0,\phi]$ and the remaining eigenphases lie
anywhere in $[0, \pi]$. Note that $B_j$ and $B_k, j\not=k$, are not necessarily disjoint.

We can write the event of
having at least one eigenphase in $[0,\phi]$ as
\begin{equation}
{\mathcal{B}} = \bigcup_{k = 1}^N B_k.
\end{equation}

For the probability of the event $B_k$ we have
\begin{equation} \label{normalizationconstant}
  \begin{split}
    P(B_k)  = \CN \int_{B_k}
    \prod_{j = 1}^N (1 -\cos\phi_j)^\alpha (1 +\cos\phi_j)^\beta
    \prod_{1 \leq j < k \leq N} (\cos\phi_j - \cos\phi_k)^2\,d \phi_j.
  \end{split}
\end{equation}

For any $k$ and any $j\ne k$ we have
\begin{equation}
P(B_k)\ = \ I(1)\ \ \quad\text{and}\quad \ \
P(B_j \cap B_k) \ =\ I(2),
\end{equation}
and in general, for $1\leq i_1 < \cdots < i_k \leq N$,
\begin{equation}
P(B_{i_1} \cap \cdots \cap B_{i_k})\ =\ I(k).
\end{equation}
By the inclusion-exclusion principle and the symmetry above,
\begin{equation}
  \begin{split}
    P({\mathcal{B}}) &\ =\ P \left(\bigcup_{k = 1}^N B_k\right) = \sum_{k = 1}^N
    (-1)^{k + 1} \sum_{1 \leq i_1 < \cdots < i_k \leq N} P(B_{i_1} \cap \cdots
    \cap B_{i_k}) \\
    &\ =\ N \cdot P(B_1) - {N \choose 2} P(B_1 \cap B_2) + {N \choose 3} P(B_1 \cap B_2 \cap B_3) \\
    &\ \ \quad\ +\cdots + (-1)^{N + 1}P(B_1 \cap \cdots \cap B_N).
  \end{split}
\end{equation}
Thus
\begin{align}
P({\mathcal{B}})\ =\ N \cdot I(1) - {N \choose 2}I(2) + {N \choose 3} I(3) + \cdots + (-1)^{N + 1} I(N).
\end{align}

Finally, the probability of having no eigenphases in $[0, \phi]^N$ is given by
\begin{align}\nonumber
P(\complement{\mathcal{B}}) & = {\CN} \int_\phi^\pi \cdots \int_\phi^\pi \prod_{j = 1}^N (1 -\cos\phi_j)^\alpha (1 +\cos\phi_j)^\beta \\ \nonumber
&\ \ \ \ \times \prod_{1 \leq j < k \leq N} (\cos\phi_j - \cos\phi_k)^2 d \phi_1 \cdots d \phi_N\\ \nonumber
& = 1 - P({\mathcal{B}})\\ \label{finally}
& = 1 - N \cdot I(1) + {N \choose 2}I(2) - {N \choose 3} I(3) + \cdots + (-1)^{N} I(N).
\end{align}

\end{proof}

%%%%%%%%%%%%%%%%%%%%%%%%%%%%%%%%%%%%%%%%%%%%%%%%%%%%%%%%%%%%%%%%%%%%%%%%%%%%%%%%%%%%%%%%%%%%%%%%%%%%%%%
%%%%%%%%%%%%%%%%%%%%%%%%%%%%%%%%%%%%%%%%%%%%%%%%%%%%%%%%%%%%%%%%%%%%%%%%%%%%%%%%%%%%%%%%%%%%%%%%%%%%%%%

\begin{proof}[Proof of Lemma \ref{lemmaIone}]
First we recall Selberg's integral (see Chapter 17 of~\cite{Mehta91}):
\begin{align}\nonumber
\mathcal{S}_{\mathcal{N}}(\rho,\eta;\gamma) &:= \int_0^1 dx_1
\cdots \int_0^1 dx_\calN \prod_{l = 1}^\calN x_l^{\rho -
1}(1-x_l)^{\eta - 1} \prod_{1 \leq j < k \leq \calN} |x_k -
x_j|^{2\gamma} \\ \label{Selberg} & = \prod_{j = 0}^{\calN - 1}
\frac{\Gamma(1 + \gamma + j\gamma) \Gamma(\rho + j\gamma)
\Gamma(\eta + j\gamma)}{\Gamma(1 + \gamma)\Gamma(\rho + \eta +
[\calN + j - 1]\gamma)}
\end{align}
and Aomoto's extension of Selberg's integral for $1 \leq R \leq \calN$
\begin{align}\nonumber
& \int_0^1 dx_1 \cdots \int_0^1 dx_\calN \prod_{j = 1}^Rx_j
\prod_{l = 1}^\calN x_l^{\rho - 1}(1-x_l)^{\eta - 1} \prod_{1 \leq
j < k \leq \calN} |(x_j - x_k)|^{2\gamma} \\ \label{Aomoto} & =
\prod_{j = 1}^R \frac{\rho + (\calN-j)\gamma}{\rho + \eta +
(2\calN -j -1)\gamma}  \prod_{j = 0}^{\calN - 1}
\frac{\Gamma(1 + \gamma + j\gamma) \Gamma(\rho + j\gamma)
\Gamma(\eta + j\gamma)}{\Gamma(1 + \gamma)\Gamma(\rho + \eta +
[\calN + j - 1]\gamma)},
\end{align}
both valid for integer $\calN$ and complex $\rho,\eta, \gamma$
with
\begin{align}
  \RRe(\rho) &> 0, &
  \RRe(\eta) &> 0, &
  \RRe(\gamma) &>
  -\min\bigg(\frac{1}{\calN}, \frac{\RRe(\rho)}{(\calN-1)},
  \frac{\RRe(\eta)}{(\calN-1)} \bigg).
\end{align}

The version of Selberg's integral we are interested in is related
to~\eqref{Selberg} by
\begin{align*}
\rho & = s + 1/2,& \eta & = r + 1/2, & \gamma & = 1,
\end{align*}
and the change of variables
\begin{equation*}
y_j = \frac{1 + \cos\phi_j}{2}
\end{equation*} as follows (note $\Gamma(2)=1$):
\begin{equation} \label{MySelberg}
  \begin{split}
    & \int_0^\pi d\phi_1 \cdots \int_0^\pi d \phi_\calN \prod_{l = 1}^\calN (1
    -\cos\phi_l)^r (1 + \cos\phi_l)^s
    \prod_{1 \leq j < k \leq \calN} (\cos\phi_j - \cos\phi_k)^2\\
    & = 2^{\calN(\calN + r + s - 1)}   \int_0^1 dy_1 \cdots \int_0^1
    dy_\calN \prod_{l = 1}^\calN (1-y_l)^{r -1/2}y_l^{s -1/2}
    \prod_{1 \leq j < k \leq \calN} (y_j - y_k)^2\\
    & = 2^{\calN(\calN + r + s - 1)}   \prod_{j=0}^{\calN-1}\frac{\Gamma(2
      +j)\Gamma(s + 1/2+j) \Gamma(r + 1/2 + j)}{ \Gamma(s + r + \calN +
      j)}.
  \end{split}
\end{equation}

The version of Aomoto's extension of our interest has parameters
\begin{align*}
\rho & = r + 1/2,& \eta & = s + 1/2, & \gamma & = 1,
\end{align*}
and we change variables
\begin{equation*}
z_j = \frac{1 - \cos\phi_j}{2}
\end{equation*} in~(\ref{Aomoto}) to obtain
{\allowdisplaybreaks
\begin{equation} \label{MyAomoto}
  \begin{split}
    & \int_0^\pi d\phi_1 \cdots \int_0^\pi d \phi_\calN \prod_{k = 1}^R (1
    -\cos\phi_k)
    \prod_{l = 1}^\calN (1 -\cos\phi_l)^r (1 + \cos\phi_l)^s  \prod_{1 \leq j < k \leq N} (\cos\phi_j - \cos\phi_k)^2\\
    & = 2^{R + \calN(\calN + r + s -1)}   \int_0^1 dz_1 \cdots \int_0^1 dz_\calN \prod_{k = 1}^R z_k \prod_{l = 1}^\calN z_l^{r -1/2}(1-z_l)^{s-1/2}   \prod_{1 \leq j < k \leq \calN} (z_j - z_k)^2 \\
    & = 2^{R + \calN(\calN + r + s -1)}
    \prod_{j = 1}^R \frac{r + 1/2 + \calN - j}{r + s + 2\calN - j}
   \prod_{j = 0}^{N - 1} \frac{\Gamma(2 + j) \Gamma(r + 1/2 +
      j) \Gamma(s + 1/2 + j)}{\Gamma(r + s + \calN + j)}.
  \end{split}
\end{equation}
}

We can now determine the normalization constant $\CN$
in~(\ref{normalizationconstant}).  We have
\begin{equation}
{\CN}^{-1} =  \int_0^\pi d \phi_1 \cdots \int_0^\pi d \phi_N \prod_{j = 1}^N (1 -\cos\phi_j)^\alpha (1 + \cos\phi_j)^\beta
\prod_{1 \leq j < k \leq N} (\cos\phi_j - \cos\phi_k)^2. \label{constantinverse}
\end{equation}
By setting $\calN = N, r = \alpha$, and $s = \beta$ in (\ref{MySelberg}) we obtain
\begin{equation}\label{constant}
{\CN}^{-1} = 2^{N(N + \alpha + \beta -1)}  \prod_{j=0}^{N-1}\frac{\Gamma(2 +j)\Gamma(\beta + 1/2+j) \Gamma(\alpha + 1/2 + j)}{ \Gamma(\alpha + \beta + N + j)}.
\end{equation}

Now, for small $\phi > 0$ and using Selberg's integral (\ref{MySelberg}) we wish to evaluate
\begin{equation} \label{Ionedefinition}
I(1) = {\CN} \tilde{I}(1)
\end{equation}
where
\begin{align*}
\tilde{I}(1)\ :=\ \int_0^\pi \cdots \int_0^\pi \int_0^\phi K(\phi_2,\dots,\phi_N) H(\phi_1,\dots,\phi_N) d \phi_1 d \phi_2 \cdots d \phi_N
\end{align*}
with
\begin{equation}\label{Kterm}
K(\phi_2,\dots,\phi_N)\ :=\ \prod_{j = 2}^N (1 -\cos\phi_j)^\alpha (1 + \cos\phi_j)^\beta \prod_{2 \leq j < k \leq N} (\cos\phi_j - \cos\phi_k)^2,
\end{equation}
and
\begin{align}\label{helpterm}
H(\phi_1,\dots,\phi_N)\ :=\ (1 -\cos\phi_1)^\alpha (1 + \cos\phi_1)^\beta \prod_{k = 2}^N (\cos\phi_1 - \cos\phi_k)^2.
\end{align}

Now we evaluate $H(\phi_1,\dots,\phi_N)$. The Taylor expansion
around $\phi_1=0$ of the first factor of $H(\phi_1,\dots,\phi_N)$
in (\ref{helpterm}) is
\begin{align} \label{Taylor1}
  \begin{split}
     (1 - \cos\phi_1)^\alpha (1 + \cos\phi_1)^\beta
    &\ =\ \left[\frac{\phi_1^{2\al}}{2^\al} - \al\frac{\phi_1^{2\al + 2}}{2^\al\cdot 12} + O(\phi_1^{2\al + 4})\right] \left[2^\beta - \beta2^{\beta-2}\phi_1^2 + O(\phi_1^4) \right]\\
    &\ =\ \frac{\phi_1^{2\al}}{2^{\al-\beta}} -
    \bigg(\frac{\al}{2^{\al-\beta}\cdot 12} + \frac{\beta}{2^{\al-\beta+2}}
    \bigg)\phi_1^{2\al + 2} + O(\phi_1^{2\al+4}).
  \end{split}
\end{align}
The Taylor expansion of the terms in the second factor defining $H(\phi_1,\dots,\phi_N)$
in~\eqref{helpterm} is
\begin{equation}\label{Taylor2A}
  \begin{split}
    \prod_{k=2}^N (\cos\phi_1 - \cos\phi_k)^2 & = \prod_{k=2}^N \left[ \bigg(1 - \frac{\phi_1^2}{2!} + O(\phi_1^4)\bigg) - \cos\phi_k \right]^2\\
    & = \prod_{k=2}^N \left[ (1-\cos\phi_k)^2 - \phi_1^2(1-\cos\phi_k)+ O(\phi_1^4) \right]\\
    & = \prod_{k=2}^N (1-\cos\phi_k)^2 - \phi_1^2 \left[\sum_{j=2}^N \prod_{k =
      2}^N \frac{(1 -\cos\phi_k)^2}{1-\cos\phi_j} \right] + O(\phi_1^4).
  \end{split}
\end{equation}
Using~\eqref{Taylor1} and \eqref{Taylor2A} in \eqref{helpterm} gives
\begin{equation}
  \begin{split}
    H(\phi_1,\dots,\phi_N) &\ =\ \left[\frac{\phi_1^{2\al}}{2^{\al-\beta}} - \bigg(\frac{\al}{2^{\al-\beta}\cdot 12} + \frac{\beta}{2^{\al-\beta+2}} \bigg)\phi_1^{2\al + 2} + O(\phi_1^{2\al+4})\right] \\
    &\quad \times \left[\prod_{k=2}^N (1-\cos\phi_k)^2 - \phi_1^2
    \left[\sum_{j=2}^N \prod_{k = 2}^N \frac{(1 -\cos\phi_k)^2}{1-\cos\phi_j}
    \right] + O(\phi_1^4)\right].
  \end{split}
\end{equation}
Multiplying out and collecting terms by powers of $\phi_1$ gives
\begin{equation}\label{Hbeforeintegrating}
  \begin{split}
    H(\phi_1,\dots,\phi_N) &\ =\ \frac{\phi_1^{2\al}}{2^{\al-\beta}} \prod_{k=2}^N (1-\cos\phi_k)^2
    - \frac{\phi_1^{2\al + 2}}{2^{\al-\beta}} \left[\sum_{j=2}^N \prod_{k = 2}^N \frac{(1 -\cos\phi_k)^2}{1-\cos\phi_j} \right] \\
    &\ \ \ \ - \bigg(\frac{\al}{2^{\al-\beta} \cdot 12} +
    \frac{\beta}{2^{\al-\beta+2}} \bigg) \phi_1^{2\al+2}\prod_{k=2}^N
    (1-\cos\phi_k)^2 + O(\phi_1^{2\al + 4}).
  \end{split}
\end{equation}
Integration of the last expression~(\ref{Hbeforeintegrating}) gives
\begin{align} \label{Hintegrated}
  \begin{split}
    \int_0^\phi H(\phi_1,\dots,\phi_N) d \phi_1  & = \prod_{k=2}^N (1-\cos\phi_k)^2
    \frac{\phi^{2\al+1}}{(2\al+1)2^{\al-\beta}} \\
    &\quad - \left[\sum_{j=2}^N \prod_{k = 2}^N \frac{(1 -\cos\phi_k)^2}{1-\cos\phi_j} \right]  \frac{\phi^{2\al+3}}{(2\al+3)2^{\al - \beta}}\\
    &\quad - \prod_{k=2}^N (1-\cos\phi_k)^2  \bigg( \frac{\al}{2^{\al -
        \beta} \cdot 12} + \frac{\beta}{2^{\al - \beta + 2}}\bigg)
    \frac{\phi^{2\al+3}}{2\al+3} + O(\phi^{2\al + 5}).
  \end{split}
\end{align}
Hence, to evaluate $\tilde{I}(1) = \int_0^\pi \cdots \int_0^\pi \int_0^\phi K(\phi_2,\dots,\phi_N) H(\phi_1,\dots,\phi_N) d \phi_1 \cdots d \phi_N$, we have to compute
\begin{align} \label{firstintegralofH}
& \int_0^\pi \cdots \int_0^\pi K(\phi_2,\dots,\phi_N) \prod_{k = 2}^N (1 - \cos\phi_k)^2 d \phi_2 \cdots d \phi_N \quad\text{and} \\
& \int_0^\pi \cdots \int_0^\pi K(\phi_2,\dots,\phi_N) \left[\sum_{j=2}^N \prod_{k = 2}^N \frac{(1 -\cos\phi_k)^2}{1-\cos\phi_j} \right] d \phi_2 \cdots d \phi_N. \label{symmetricpartofH}
\end{align}

Observe that the integrand of \reff{symmetricpartofH} is symmetric in its variables $\phi_2, \ldots, \phi_N$. Therefore we have
\begin{multline} \label{symmetricpartofHA}
\int_0^\pi \cdots \int_0^\pi K(\phi_2,\dots,\phi_N)   \left[\sum_{j=2}^N \prod_{k = 2}^N \frac{(1 -\cos\phi_k)^2}{1-\cos\phi_j}\right]
d \phi_2 \cdots d \phi_N \\
= (N-1) \int_0^\pi \cdots \int_0^\pi K(\phi_2,\dots,\phi_N)   (1 -\cos\phi_2) \prod_{k=3}^N (1-\cos\phi_k)^2 d \phi_2 \cdots d \phi_N.
\end{multline}
Evaluating the integral \reff{firstintegralofH} yields
\begin{multline}
   \int_0^\pi \cdots \int_0^\pi K(\phi_2,\dots,\phi_N) \prod_{k=2}^N (1-\cos\phi_k)^2 d \phi_2 \cdots d \phi_N \\
    = \int_0^\pi \cdots \int_0^\pi \prod_{j = 2}^N (1 -\cos\phi_j)^{\al+2} (1 +\cos\phi_j)^\beta
    \prod_{2 \leq j < k \leq N} (\cos\phi_j - \cos\phi_k)^2 d
    \phi_2 \cdots d \phi_N,
\end{multline}
and using Selberg's integral (\ref{MySelberg}) with $\calN = N-1, r = \al + 2,$ and $s = \beta$ gives
\begin{multline}\label{firstintegral}
\int_0^\pi \cdots \int_0^\pi K(\phi_2,\dots,\phi_N)  \prod_{k=2}^N (1-\cos\phi_k)^2 d \phi_2 \cdots d \phi_N \\
= 2^{(N-1)(N + \al + \beta)}
\prod_{j =0}^{N - 2}\frac{\Gamma(2+j) \Gamma(\beta + 1/2 + j) \Gamma(\al + 5/2 +j) }{\Gamma(\al + \beta + N + 1 + j)}.
\end{multline}
Normalizing the last expression (\ref{firstintegral}) with ${\CN}$
from~(\ref{constantinverse}) we obtain
\begin{multline} \label{firstintegralendresult} {\CN}
  \int_0^\pi \cdots \int_0^\pi K(\phi_2,\dots,\phi_N)   \prod_{k=2}^N (1-\cos\phi_k)^2 d \phi_2
  \cdots d \phi_N
  \\
  = \frac{1}{2^{\al + \beta}}   \frac{\Gamma(\al + N + 1/2) \Gamma(\al +
    \beta + N)}{\Gamma(\al + 1/2) \Gamma(\al + 3/2) \Gamma(N + 1) \Gamma(\beta +
    N-1/2)}.
\end{multline}

We now evaluate the integral in \reff{symmetricpartofHA}. For this we
note that
\begin{equation}
(1 -\cos\phi_2) \prod_{k=3}^N (1-\cos\phi_k)^2\ =\ \prod_{k = 2}^N (1 -\cos\phi_k)
 \prod_{k = 3}^N (1- \cos\phi_k).
\end{equation}

So
{\allowdisplaybreaks
\begin{equation}
  \begin{split}
    \int_0^\pi &\cdots \int_0^\pi K(\phi_2,\dots,\phi_N) (1 -\cos\phi_2) \prod_{k=3}^N
    (1-\cos\phi_k)^2\, d \phi_2 \cdots d \phi_N
    \\
    & = \int_0^\pi \cdots \int_0^\pi K(\phi_2,\dots,\phi_N)  \prod_{k = 2}^N (1 -\cos\phi_k)
    \prod_{k = 3}^N (1- \cos\phi_k) d \phi_2 \cdots d \phi_N
    \\
    & = \int_0^\pi \cdots \int_0^\pi \prod_{j = 2}^N (1 -\cos\phi_j)^\alpha (1 + \cos\phi_j)^\beta \prod_{2 \leq j < k \leq N} (\cos\phi_j - \cos\phi_k)^2 \\
    &\quad \times \prod_{k = 2}^N (1 -\cos\phi_k)\prod_{k = 3}^M (1- \cos\phi_k) d \phi_2 \cdots d \phi_N\\
    & = \int_0^\pi \cdots \int_0^\pi \prod_{k=3}^N (1-\cos\phi_k)
    \prod_{j = 2}^N (1 -\cos\phi_j)^{\al + 1}
    (1 +\cos\phi_j)^\beta \\
    &\quad \times \prod_{2 \leq j < k \leq N} (\cos\phi_j - \cos\phi_k)^2 d
    \phi_2 \cdots d \phi_N.
  \end{split}
\end{equation}
} With $R = N - 2, \calN = N - 1, r = \al + 1$, and $s = \beta$ in Aomoto's
integral~(\ref{MyAomoto}) this evaluates to
\begin{equation}\label{secondintegral}
  \begin{split}
    \int_0^\pi \cdots& \int_0^\pi K(\phi_2,\dots,\phi_N) (1 -\cos\phi_2) \prod_{k=3}^N (1-\cos\phi_k)^2 d \phi_2 \cdots d \phi_N \\
    & = 2^{(N-2) + (N-1)(N + \al + \beta - 1)}
    \prod_{j = 1}^{N - 2} \frac{\al + N + 1/2 -j}{\al + \beta + 2N - 1 - j}\\
    &\quad \times
    \prod_{j = 0}^{N - 2}\frac{\Gamma(2 + j) \Gamma(\al + 3/2 + j) \Gamma(\beta + 1/2 + j)}{\Gamma(\al+ \beta + N +j)}\\
    & = 2^{(N-2) + (N-1)(N + \al + \beta - 1)} \frac{\Gamma(\al + N +
      1/2)}{\Gamma(\al + 5/2)} \frac{\Gamma(\al + \beta + N + 1)}{\Gamma(\al +
      \beta + 2N -1)}
    \frac{\Gamma(\al + N -1/2)}{\Gamma(\al + 1/2)} \\
    &\quad \times \prod_{j = 0}^{N - 2}\frac{\Gamma(2+j)\Gamma(\al + 3/2+
      j)\Gamma(\beta + 1/2 + j)}{\Gamma(\al + \beta + N +j)}.
  \end{split}
\end{equation}
Normalizing the last expression (\ref{secondintegral}) with ${\CN}$ from (\ref{constantinverse})
we obtain
\begin{equation}\label{secondintegralendresult}
  \begin{split}
    {\CN} & \int_0^\pi \cdots \int_0^\pi K(\phi_2,\dots,\phi_N)  (1
    -\cos\phi_2) \prod_{k=3}^N (1-\cos\phi_k)^2 d\phi_1 \cdots d \phi_N
    \\
    & = 2^{-\al - \beta - 1} \frac{\Gamma(\al + N + 1/2) \Gamma(\al +
      \beta + N + 1)}{\Gamma(N + 1) \Gamma(\beta + N-1/2) \Gamma(\al + 1/2)
      \Gamma(\al + 5/2)}.
  \end{split}
\end{equation}

Putting (\ref{firstintegralendresult}) and (\ref{secondintegralendresult}) into (\ref{Ionedefinition}) we obtain
\begin{equation}
  \begin{split}
    I(1) &\ =\ \widetilde{H}_1 \frac{\phi^{2\al+1}}{(2\al+1)2^{\al - \beta}} - (N-1) \widetilde{H}_2 \frac{\phi^{2\al+3}}{(2\al +3)2^{\al - \beta}} \\
    &\quad\ \ -\ \bigg(\frac{\al}{2^{\al - \beta}\cdot 12} + \frac{\beta}{2^{\al -
        \beta + 2}} \bigg) \widetilde{H}_1 \frac{\phi^{2\al +3}}{2\al+3} +
    O(\phi^{2\al+5})
  \end{split}
\end{equation}
where
\begin{equation}
\widetilde{H}_1\ :=\ \frac{1}{2^{\al + \beta}}
\frac{\Gamma(\al + N + 1/2) \Gamma(\al + \beta + N)}{\Gamma(\al + 1/2) \Gamma(\al + 3/2) \Gamma(N + 1) \Gamma(\beta + N - 1/2)}
\end{equation}
and
\begin{equation}
\widetilde{H}_2\ :=\ \frac{1}{2^{\al + \beta + 1}}
\frac{\Gamma(\al + N + 1/2) \Gamma(\al + \beta + N + 1)}{\Gamma(N + 1) \Gamma(\beta + N - 1/2) \Gamma(\al + 1/2) \Gamma(\al + 5/2)}.
\end{equation}

Finally, we rewrite the result slightly and get
\begin{equation}
\begin{split}
I(1) &\ =\ H_1 \frac{\phi^{2\al+1}}{2\al+1} - (N-1) H_2 \frac{\phi^{2\al+3}}{2\al +3}  - \bigg(\frac{\al}{12} + \frac{\beta}{4} \bigg) H_1 \frac{\phi^{2\al +3}}{2\al+3} + O(\phi^{2\al+5}) \\
&\ =\ H_1 \frac{\phi^{2\al+1}}{2\al+1} - \left[(N-1) H_2 + \bigg(\frac{\al}{12} + \frac{\beta}{4} \bigg) H_1 \right]
\frac{\phi^{2\al+3}}{2\al+3} + O(\phi^{2\al+5})
\end{split}
\end{equation}
where
\begin{equation}
H_1\ :=\ \frac{\Gamma(\al + N + 1/2) \Gamma(\al + \beta + N)}{2^{2\al}\Gamma(\al + 1/2) \Gamma(\al + 3/2) \Gamma(N + 1) \Gamma(\beta + N - 1/2)}
\end{equation}
and
\begin{equation}
H_2\ :=\   \frac{\Gamma(\al + N + 1/2) \Gamma(\al + \beta + N + 1)}{2^{2\al+1}\Gamma(N + 1) \Gamma(\beta + N - 1/2) \Gamma(\al + 1/2) \Gamma(\al + 5/2)}.
\end{equation}
This completes the proof of Lemma \ref{lemmaIone}.

\end{proof}

%%%%%%%%%%%%%%%%%%%%%%%%%%%%%%%%%%%%%%%%%%%%%%%%%%%%%%%%%%%%%%%%%%%%%%%%%%%%%%%%%%%%%%%%%%%%%%%%%%%%%%
%%%%%%%%%%%%%%%%%%%%%%%%%%%%%%%%%%%%%%%%%%%%%%%%%%%%%%%%%%%%%%%%%%%%%%%%%%%%%%%%%%%%%%%%%%%%%%%%%%%%%%
\begin{proof}[Proof of Lemma \ref{lemmaIH}]
The definition of $I(n)$ is
\begin{equation}
  \begin{split}
    I(n)  & \ =\  {\CN} \underbrace{\int_0^\pi \cdots \int_0^\pi}_{N - n \text{~times}} \underbrace{\int_0^\phi \cdots \int_0^\phi}_{n \text{~times}} \prod_{j = 1}^N (1 -\cos\phi_j)^\alpha (1 -\cos\phi_j)^\beta \\
    &\quad \ \ \times
  \ \prod_{1 \leq j < k \leq N} (\cos\phi_j - \cos\phi_k)^2 d
    \phi_1 \cdots d \phi_N.
  \end{split}
\end{equation}
Here we are only interested in the size of $I(n)$ in terms of $n$, so we can disregard the
normalization constant ${\CN}$. For $n \ge 2$ we consider
\begin{equation*}
\tilde{I}(n)\ :=\ {\CN}^{-1} I(n).
\end{equation*}
Then we have
{\allowdisplaybreaks
\begin{align} \label{tildeIh}
  \begin{split}
    \tilde{I}(n) & = \int_0^\pi d \phi_{n+1} \cdots \int_0^\pi d \phi_N \prod_{j = n+1}^N (1 -\cos\phi_j)^\alpha (1 + \cos\phi_j)^\beta \\
    &\quad \times \prod_{n+1 \leq j < k \leq N} (\cos\phi_j - \cos\phi_k)^2\\
    &\quad \times \int_0^\phi \bigg\{(1 - \cos\phi_n)^\alpha (1 + \cos\phi_n)^\beta \prod_{j = n + 1}^N (\cos\phi_n - \cos\phi_j)^2 \\
    &\quad \times \int_0^\phi (1 - \cos\phi_{n-1})^\alpha (1 + \cos\phi_{n-1})^\beta \prod_{j = n + 1}^N (\cos\phi_{n-1} - \cos\phi_j)^2\\
    &\quad \times (\cos\phi_{n-1} - \cos\phi_n)^2 d\phi_{n - 1} \\
    & \quad\  \vdots\\
    &\quad \times \int_0^\phi (1 - \cos\phi_1)^\alpha (1 + \cos\phi_1)^\beta \prod_{j = n + 1}^N (\cos\phi_1 - \cos\phi_j)^2\\
    &\quad \times (\cos\phi_1 - \cos\phi_2)^2 \times (\cos\phi_1 -
    \cos\phi_3)^2 \times \cdots \times (\cos\phi_1 - \cos\phi_n)^2
    d\phi_1\bigg\} d \phi_n.
  \end{split}
\end{align}
}

Now for $j,k \leq n$, $j \ne k$ and $\phi_j, \phi_k \in [0,\phi]$ for small
$\phi > 0$ we have
\begin{equation} \label{cosinusinequality}
(\cos\phi_j - \cos\phi_k)^2 \leq (1 -\cos \phi)^2.
\end{equation}
There are $(n-1) + \cdots + 1 = n(n-1)/2$ terms of the form $(\cos\phi_j - \cos\phi_k)^2$
(with $j =1,\ldots,n-1$, and $k = j + 1, \ldots, n$) occurring in
(\ref{tildeIh}), each of which we bound using~(\ref{cosinusinequality}).

From equation (\ref{Hintegrated}) in the proof of Lemma \ref{lemmaIone} we
derive for $k = 1, \ldots, n$ that
{\allowdisplaybreaks
\begin{equation} \label{integralidentity}
  \begin{split}
    \int_0^\phi &(1 - \cos\phi_k)^\alpha (1 + \cos\phi_k)^\beta \prod_{j = n + 1}^N (\cos\phi_k - \cos\phi_j)^2 d\phi_k \\
    & = \prod_{j = n+1}^N ( 1 - \cos\phi_j)^2  \frac{\phi^{2\al +
        1}}{2^{\al - \beta}(2\al+1)} + O(\phi^{2\al + 3}).
  \end{split}
\end{equation}
}
Using (\ref{cosinusinequality}) and (\ref{integralidentity}) we obtain
{\allowdisplaybreaks
\begin{equation}
  \begin{split}
    \tilde{I}(n) & \leq \int_0^\pi d \phi_{n+1} \cdots \int_0^\pi d \phi_N \prod_{j = n+1}^N (1 -\cos\phi_j)^\alpha (1 + \cos\phi_j)^\beta \\
    &\quad \times \prod_{n+1 \leq j < k \leq N} (\cos\phi_j - \cos\phi_k)^2 (1 -\cos \phi)^{n(n-1)}\\
    &\quad \times \left[\prod_{j = n+1}^N ( 1 - \cos\phi_j)^2 \times \frac{\phi^{2\al + 1}}{2^{\al - \beta}(2\al+1)} + O(\phi^{2\al + 3})   \right]^n\\
%%%%%%%%%%%%%%%%%%%%%%%%%%%%%%%%%%%%%%%%%%%%%%%%%%%%%%%%%%%%%%
    & = \int_0^\pi d \phi_{n+1} \cdots \int_0^\pi d \phi_N \prod_{j = n+1}^N (1 -\cos\phi_j)^\alpha (1 + \cos\phi_j)^\beta\\
    &\quad \times \prod_{n+1 \leq j < k \leq N} (\cos\phi_j - \cos\phi_k)^2 (1 -\cos \phi)^{n(n-1)}\\ &\quad \times \left[\prod_{j = n+1}^N ( 1 - \cos\phi_j)^{2n} \times \frac{\phi^{(2\al + 1)n}}{[2^{\al - \beta}(2\al+1)]^n} + O(\phi^{n(2\al+1) + 2})\right]\\
%%%%%%%%%%%%%%%%%%%%%%%%%%%%%%%%%%%%%%%%%%%%%%%%%%%%%%%%%%%%%%
    & = \int_0^\pi d \phi_{n+1} \cdots \int_0^\pi d \phi_N \prod_{j = n+1}^N (1
    -\cos\phi_j)^{\al+2n}
    (1 + \cos\phi_j)^\beta \\
    &\quad \times \prod_{n+1 \leq j < k \leq N} (\cos\phi_j - \cos\phi_k)^2 (1 -\cos \phi)^{n(n-1)}\\
    &\quad \times \left[\frac{\phi^{(2\al + 1)n}}{[2^{\al - \beta}(2\al+1)]^n}
    + O(\phi^{n(2\al+1) + 2})\right].
  \end{split}
\end{equation}
}
Hence for some constants $c_1$ and $c_2$
{\allowdisplaybreaks
\begin{equation}
  \begin{split}
    \tilde{I}(n) &\ \leq\  c_1   (1 -\cos \phi)^{n(n-1)} \left[\phi^{(2\al + 1)n} + O(\phi^{n(2\al+1) + 2})\right]\\
    &\ =\ c_2 \left[\phi^2 + O(\phi^4)\right]^{n(n-1)}\left[\phi^{(2\al + 1)n} + O(\phi^{n(2\al+1) + 2})\right]\\
    &\ =\ c_2 \left[\phi^{2n(n-1)} + O(\phi^{2n(n-1) + 2}) \right]\left[\phi^{(2\al + 1)n} + O(\phi^{n(2\al+1) + 2})\right]\\
    &\ =\ c_2 \phi^{(2\al + 1)n + 2n(n - 1)} + O(\phi^{(2\al+1)n + 2n(n-1) + 2})\\
    &\ =\ c_2 \phi^{2\al n + 2n^2 - n} + O(\phi^{(2\al+1)n + 2n(n-1) + 2}).
  \end{split}
\end{equation}
}
Hence $\tilde{I}(n) \ll \phi^{2\al n + 2n^2 - n}$, which implies $I(n) \ll \phi^{2\al n +
  2n^2 - n}$ as required to finish the proof.
\end{proof}

\begin{remark} \rm
It is natural to ask whether we can determine further terms of the Taylor expansion of $\ENf$
than provided in Proposition \ref{InitialConditionProposition}. By Lemma \ref{lemmaEGN} this requires us
to determine $I(n)$ with $n\geq 2$. However, for $I(n)$ with $n \geq 2$ we encounter multiple integrals
which cannot be dealt with using Selberg's or Aomoto's integral. To our knowledge the required
generalization of Selberg's integral does not exist.
\end{remark}

%%%%%%%%%%%%%%%%%%%%%%%%%%%%%%%%%%%%%%%%%%%%%%%%%%%%%%%%%%%%%%%%%%%%%%%%%%%%%%%%%%%%%%%%%%%%%%%%%%%%%%%%%%%%%%%%
%%%%%%%%%%%%%%%%%%%%%%%%%%%%%%%%%%%%%%%%%%%%%%%%%%%%%%%%%%%%%%%%%%%%%%%%%%%%%%%%%%%%%%%%%%%%%%%%%%%%%%%%%%%%%%%%
%%%%%%%%%%%%%%%%%%%%%%%%%%%%%%%%%%%%%%%%%%%%%%%%%%%%%%%%%%%%%%%%%%%%%%%%%%%%%%%%%%%%%%%%%%%%%%%%%%%%%%%%%%%%%%%%

\section{First method: Initial conditions for the Painlev\'e equation}
\label{Matlabcode}

Here we use the results from the previous section to actually state the initial conditions \reff{InitialConditionsMatLab} for our
system of differential
equations \reff{Painlevesystem} which gives the distribution $\nuN(\phi)$ of the first eigenphase \reff{distributionoffirsteigenphases}
of random matrices in the Jacobi ensemble $J_N=J_N^{(\alpha,\beta)}$. With these initial conditions we then implement a MATLAB algorithm
to compute a numerical approximation for $\nuN(\phi)$. The full MATLAB code is provided in the appendix. For its implementation we follow
some of the ideas in Edelman and Persson~\cite{Edel06}.

As outlined at the end of Section \ref{Method1-PainleveEq} we are left to provide $\tEN(t_0), h(t_0),$ and $h'(t_0)$, namely
the  initial conditions \reff{InitialConditionsMatLab},
for some $t_0 = 1 -\varepsilon$ with small $\varepsilon > 0$.
Recall that, according to Proposition~\ref{InitialConditionProposition}, we have
\begin{equation} \label{EGNforMatlab}
  \EN(\phi) = 1 - N \bigg(H_1 \frac{\phi^{2\al+1}}{2\al+1} -
    \left[(N-1) H_2 + \bigg(\frac{\al}{12} + \frac{\beta}{4} \bigg) H_1 \right]
    \frac{\phi^{2\al +3}}{2\al+3}\bigg)  + O(\phi^{2\al+4}),
\end{equation}
with
\begin{equation} \label{HoneforMatlab}
H_1 = \frac{\Gamma(\al + N + 1/2) \Gamma(\al + \beta + N)}{2^{2\al}\Gamma(\al + 1/2) \Gamma(\al + 3/2) \Gamma(N + 1) \Gamma(\beta + N - 1/2)}
\end{equation}
and
\begin{equation} \label{HtwoforMatlab}
H_2 = \frac{\Gamma(\al + N + 1/2) \Gamma(\al + \beta + N + 1)}{2^{2\al+1}\Gamma(N + 1) \Gamma(\beta + N - 1/2) \Gamma(\al + 1/2) \Gamma(\al + 5/2)}.
\end{equation}

Via the substitution $\phi = \cos^{-1}(2t-1)$, equation~\reff{EGNforMatlab}
provides a good approximation for $\tEN(t_0)$.
In view of the definition~(\ref{Okamoto}) of the auxiliary Hamiltonian $h$, we need to
differentiate $\tEN(t)$ twice with respect to $t$ in order to obtain the initial
conditions $h(t_0)$, $h'(t_0)$ (the second and third components of
\reff{InitialConditionsMatLab}).

Since $\phi=\cos^{-1}(2t-1)$ implies $d \phi / d t = -1/\sqrt{t(1-t)}$ and
$\tEN(t)=\EN(\phi)$, the chain rule and equation \reff{Okamoto} give
\begin{align}
%\begin{split}
 \nonumber
  h(t) & = t \cdot e_2'[\bb] - \tfrac{1}{2}e_2[\bb] + t(t-1)\frac{\tfrac{d}{dt}\tEN(t)}{\tEN(t)} \\ \label{conditiontwoMatlab}
  & = t \cdot e_2'[\bb] - \tfrac{1}{2}e_2[\bb] +
  \sqrt{t(1-t)} \frac{{\EN}'(\cos^{-1}(2t-1))}{\EN(\cos^{-1}(2t-1))}
%\end{split}
\end{align}
and
{\allowdisplaybreaks
\begin{align}
  \nonumber
    h'(t) & = e_2'[\bb] + \frac{1-2t}{2\sqrt{t(1-t)}}
    \frac{{\EN}'(\cos^{-1}(2t-1))}{\EN(\cos^{-1}(2t-1))}\\ \nonumber
   &\quad + \frac{\sqrt{t(1-t)}}{\EN(\cos^{-1}(2t-1))}
 \left( \frac{-{\EN}''(\cos^{-1}(2t-1))}{\sqrt{t(1-t)}}  \right) \\ \label{conditionthreeMatlab}
    &\quad+ \sqrt{t(1-t)}{\EN}'(\cos^{-1}(2t-1) )
    \left(\frac{{\EN}'(\cos^{-1}(2t-1) )}{\sqrt{t(1-t)} \EN(\cos^{-1}(2t-1)
        )^2}\right) \\ \nonumber
  & = e_2'[\bb] + \frac{1-2t}{2\sqrt{t(1-t)}}
    \frac{{\EN}'(\cos^{-1}(2t-1))}{\EN(\cos^{-1}(2t-1))}\\ \nonumber
   &\quad - \frac{{\EN}''(\cos^{-1}(2t-1))}{\EN(\cos^{-1}(2t-1))}
    + \frac{{\EN}'(\cos^{-1}(2t-1) )^2}{\EN(\cos^{-1}(2t-1)
        )^2}.
\end{align}
}
With \reff{EGNforMatlab}, \reff{conditiontwoMatlab} and
\reff{conditionthreeMatlab} we have the initial conditions \reff{InitialConditionsMatLab} for
\reff{Painlevesystem}. As mentioned earlier, with these initial conditions we can now implement
a MATLAB algorithm to compute the distribution of the first eigenvalue of random matrices in the Jacobi
ensemble. We provide the full MATLAB code in the appendix. Also, it can be obtained from the authors or
from their web pages, such as
\begin{center}\url{http://www.maths.bris.ac.uk/~mancs/publications.html}.\end{center}

%%%%%%%%%%%%%%%%%%%%%%%%%%%%%%%%%%%%%%%%%%%%%%%%%%%%%%%%%%%%%%%%%%%%%%%%%%%%%%%%%%%%%%%%%%%%%%%%%%%%%%%%%%%%%%%%
%%%%%%%%%%%%%%%%%%%%%%%%%%%%%%%%%%%%%%%%%%%%%%%%%%%%%%%%%%%%%%%%%%%%%%%%%%%%%%%%%%%%%%%%%%%%%%%%%%%%%%%%%%%%%%%%
%%%%%%%%%%%%%%%%%%%%%%%%%%%%%%%%%%%%%%%%%%%%%%%%%%%%%%%%%%%%%%%%%%%%%%%%%%%%%%%%%%%%%%%%%%%%%%%%%%%%%%%%%%%%%%%%

\section{Second method: Symbolic solution using power series}
\label{sec:power-series}
In this section we describe an algorithm to compute the power series expansion
of the Painlev\'e function $h(t)$ at $t=0$, leading to the numerical computation
of $\EN(\phi)$ for $\phi$ close to~$\pi$. It is rather unfortunate (for our
intended application) that $t=1$ is a branch-point singularity of~$h(t)$, and as
a consequence a power-series expansion of~$h$ about $t=1$ is a Puisseaux series
(i.~e., a series in fractional powers of~$t$), at least if the parameters $a,b$
(and, eventually, $N$) are rational numbers; for arbitrary values of the
parameters the situation would be even more complicated. Therefore, we content
ourselves for the time being with finding the power series expansion about
$t=0$.

The idea of the algorithm is very simple: The coefficients $h_0$, $h_1$ of the
expansion $h(t)=h_0+h_1t + h_2t^2+\dots$ are given in
equation~\eqref{t=0-BoundCond}. These are used to bootstrap a recursive search
for the higher coefficients $h_2$, $h_3$, \dots, regarding each unknown $h_k$ as
implicitly defined by the earlier coefficients $h_0$, $h_1$, $h_2$, \dots,
$h_{k-1}$ through the Painlev\'e equation~\eqref{Painleve6}. Given the
complicated nonlinear nature of the Painlev\'e equation it is not immediately
obvious that this approach will work in practice, but fortunately it does, and
each successive coefficient $h_k$ is expressible as a rational function of the
previous ones, hence ultimately as a rational function of the
parameters~$a,b,N$. Once many terms are computed, the power series for
$h(t)$ can be used to evaluate~$\tEN(t)$ by solving for the latter in
equation~\eqref{Okamoto}; then we use the change of variables $t\to\phi$ and
differentiation to compute the density $\nuN(\phi)$ of the distribution of the
first eigenvalue, at least for $\phi$ relatively close to~$\pi$.

We seek to find the coefficients~$h_k$ as exact rational numbers; for this
reason, we will work with exact (truncated) power series with rational
coefficients. This approach has the enormous advantage that the complicated
operations needed to evaluate both sides of the Painlev\'e
equation~\eqref{Painleve6} introduce no numerical errors at all in the
evaluation of successive higher coefficients. The price paid is a more expensive
calculation compared to one done using exclusively floating-point
arithmetic. Remarks on the choice of the number of needed terms to reach
adequate numerical precision follow below in Section~\ref{sec:numer-painl-solv}.

In the appendix we list the code for an implementation of this algorithm in
SAGE~\cite{SAGE} (\emph{Software for Algebra and Geometry Exploration}), a free
and open-source computer algebra system, although Maxima~\cite{MAXIMA} plays an
important role behind the scenes. The Python syntax underlying SAGE is clean and
the code listing should prove useful both SAGE newcomers and those interested
in porting it to other computer algebra systems.  Line numbers from
the SAGE code listing will be referenced below as needed.

We take the parameters $a$, $b$ and~$N$ to be (fixed) rational numbers, and
regard the function
\begin{equation}
  \label{eq:PowSerTau}
  h(t) = h_0 + h_1t + h_2t^2 + \cdots
\end{equation}
as having coefficients which are rational numbers. In particular, from
equation~\eqref{t=0-BoundCond} we have
\begin{align}
  \label{eq:h0}
  h_0 &= -\frac{1}{2}e_2[\bb] -N(b + N) \\
\label{eq:h1}
h_1 &= e_2'[\bb] + \frac{N(N+b)(2N + a + b)}{2N + b}
\end{align}
with $e_2[\bb]$, $e_2'[\bb]$ given in terms of $a,b,N$
by~\eqref{definitionofbs}, \eqref{e2} and~\eqref{e2p}.

In the SAGE code listing, the values of the basic parameters $a,b,N$ are
hard-coded, as well as the maximum degree \texttt{DEGREE} of precision of all
(truncated) power series (lines 1--4). Note that the algorithm's implementation
depends crucially on inputting rational values for the parameters $a,b,N$. With
a limitation explained below, the algorithm can handle integral as well as
rational values of the parameter $\mathtt{n}=N$.

When, at any given point, the coefficients $h_0,\dots,h_{k-1}$ are known, but
$h_k$ is still unknown, we wish to regard the latter as an indeterminate, say
$h_k=X$. Let
\begin{equation}
  \label{eq:hTrunc}
  h(t) = h_0 + h_1t + \dots + h_{k-1}t^{k-1} + Xt^k.
\end{equation}
Then $h$ lies in the ring $\mathbf{Q}[X,t]$ of polynomials in $X,t$ with
rational coefficients. Let us denote by $LHS(h)$ and $RHS(h)$ the polynomials
obtained by substitution of~\eqref{eq:hTrunc} in the Painlev\'e
equation~\eqref{Painleve6}, and let $PEZ(h)=RHS(h)-LHS(h)$
(``\emph{P}ainlev\'e-\emph{E}qual-to-\emph{Z}ero''---lines 33--37). The natural
hope is that if $h_0, h_1,\dots,h_{k-1}$ are chosen correctly, then $PEZ(h) =
p_k(X)t^k + O(t^{k+1})$ for a non-constant polynomial $p_k\in \mathbf{Q}[x]$ and
some (unspecified) polynomial $O(t^{k+1})$ divisible by $t^{k+1}$. Then $h_k$
should be chosen to be a root of $p_k(X)$.

Performing exact polynomial arithmetic in $\mathbf{Q}[X,t]$ to
compute $PEZ(h)$ is quite expensive, especially since we only need
to determine the coefficient $p(X)$ of the lowest power of~$t$.
For computational purposes, however, it is enough to regard $h$ as
a finite truncation of an infinite power series in~$t$,
systematically neglecting any higher-order terms not needed for
the immediate purpose at hand---namely the determination of
$p_k(X)$. Fortunately, SAGE can do algebra in power series rings
(with help from Maxima).

Henceforth we work in the ring $\mathbf{S}=\mathbf{Q}[X][[t]]$ of power series
in~$t$ whose coefficients are in $\mathbf{Q}[X]$ (polynomials in $X$ with
rational coefficients) as done in lines 13--14 of the SAGE code. In order to
determine $p_k(X)$ it is enough to know $h$ up to $O(t^{k+2})$ terms. (It is not
quite enough to work modulo $t^{k+1}$ because $p_k(X)$ \emph{a priori} depends
also $h_{k+1}$, not just on the currently unknown $X=h_k$; luckily, the solution
found \emph{a posteriori} shows this dependence to be fictitious.)

We thus set
\begin{equation}
  \label{eq:hTrunc2}
  h(t) = h_0 + h_1t + \dots + h_{k-1}t^{k-1} + Xt^k + O(t^{k+2})
\end{equation}
and compute $PEZ(h) = p_k(X)t^k + O(t^{k+1})$ for a \emph{linear}
polynomial~$p_k$ (the only exception is $p_2$, which is quadratic with a trivial
root $X=0$). In the SAGE implementation, we store the coefficients $h_k$ in a
SAGE list \texttt{g} (lines 26 \&~27---it is here that crucial use is made of
the initial conditions~\eqref{eq:h0} and~\eqref{eq:h1}).

Successively (main loop in lines 31--46), for each $k=2,3,\dots$ it suffices to
take $h_k$ as the unique (nontrivial) root of $p_k(X)$, which is a rational
number. (Note that this is the only point in the algorithm at which symbolic
algebra is needed, and that because of the trivial nature of the equation solved
it would be easy to write an implementation dispensing with any use of symbolic
algebra, a task which we presently avoid simply because the resulting code would
be longer and more difficult to read.)  This root is found in lines 38 \&~39,
then appended to the coefficient list in line~40 and $h(t)$ reconstructed using
the newly-found $h_k$ in lines 42--45.

A couple of remarks on the code are in order. SAGE does not seem to understand
that we wish to interpret the variable \texttt{x} of the power series ring
\texttt{F} to be a ``symbolic'' variable with respect to which the equation
\texttt{PEZ[i] == 0} (i.~e., $p_i(X)=0$) is to be solved. We therefore need
explicitly to replace the ring's variable \texttt{x} with a symbolic variable
\texttt{X} (lines 29 \&~38) before finding the root of \texttt{PEZ[i]}, which is
then appended at the end of the coefficient list~\texttt{g} (lines 39 \&~40). We
also found it simpler to reconstruct \texttt{h} from scratch in lines 42--45
using the coefficient list~\texttt{g} than figuring out a way both to
(i)~increase the precision of~\texttt{h}, and (ii)~substitute the newly-found
rational coefficient~\texttt{g[i]} for~\texttt{x}. (The wrapper \texttt{QQ(...)}
around the argument of~\texttt{g} is used to convert (``coerce'' in SAGE lingo)
the root found by Maxima to a \emph{bona fide} SAGE rational number.)

Solving for $\tEN(t)$ in the definition~\eqref{Okamoto} of the auxiliary
Hamiltonian $h(t)$ yields
\begin{equation} \label{EN-from-h}
  \begin{split}
    \tEN(t) &= \exp\left(\int\frac{h(t) - t \cdot e_2'[\bb] +
        \tfrac{1}{2}e_2[\bb]}{t(t-1)}\,dt + c\right) \qquad\text{for a suitable constant $c$} \\
    &= C\exp\left(\int\frac{h_{0}+\tfrac{1}{2}e_2[\bb] + t\cdot(h^{(1)}(t) -
        e_2'[\bb])} {t(t-1)}\,dt\right) \\
    &\qquad\text{where $C=e^c$ and $h^{(1)}(t) = (h(t)-h_0)/t = h_1+h_2t + h_3t^2+\dots$}\\
    &= C\exp\left(\int\frac{-N(N+b)} {t(t-1)}\,dt + \int \frac{h^{(1)}(t) -
        e_2'[\bb]}{t-1}\,dt \right) \\
    &\qquad\text{since $h_0=-\frac{1}{2}e_2[\bb] -N(N+b)$
      from~\eqref{t=0-BoundCond}} \\
    &= C\exp\left(-N(N+b)\int \left( \frac{1} {t-1}-\frac{1}{t}\right)\,dt +
      \int \frac{h^{(1)}(t) - e_2'[\bb]}{t-1}\,dt \right) \\
    &= C\exp\left(N(N+b)\log t
      + \int \frac{h^{(1)}(t) - e_2'[\bb]-N(N+b)}{t-1}\,dt \right) \\
    &= C t^{N(N+b)}\exp\left(\int \frac{h^{(1)}(t) - e_2'[\bb]-N(N+b)}{t-1}\,dt
    \right).
  \end{split}
\end{equation}
Note that the integrand in the last expression above is regular at $t=0$. We define
\begin{equation}\label{eq:F}
  \mathcal{F}(t)
  = \exp\left(\int_0^t \frac{h^{(1)}(\tau) - e_2'[\bb]-N(N+b)}{\tau-1}\,d\tau \right).
\end{equation}
Then $\mathcal{F}(t)$ has a series expansion (in integral powers) about~$0$ with
$\mathcal{F}(0)=1$, and equation~\eqref{EN-from-h} reads
\begin{equation}
  \label{eq:ENF}
  \tEN(t) = Ct^{N(N+b)}\mathcal{F}(t).
\end{equation}
It follows that the leading-order term of the power series expansion of
$\tEN(t)$ about $t=0$ is $Ct^{N(N+b)}$, explaining the nomenclature
\texttt{leadexp} (``leading exponent'') for the auxiliary SAGE variable defined
in line~25.

The value of the constant $C$ in~\eqref{eq:ENF} can be read off from
equation~(3.5) in~\cite{ForrWitte04} as the quotient of the normalization
constants $\tCN=1/\mathcal{S}_N(a+1,b+1;1)$ for the Jacobi
ensembles $J_N^{{(a,b)}}$ and $\widetilde C_N^{(0,b)}=1/\mathcal{S}_N(1,b+1;1)$ for
$J_N^{(0,b)}$ (recall that we swap the role of $a$ and $b$ relative to
Forrester and Witte; moreover, $\tCN = \widetilde C_N^{(b,a)}$). Explicitly,
\begin{equation}
  \label{eq:C}
  C = \frac{\tCN}{\widetilde C_N^{(0,b)}}
   = \frac{\mathcal{S}_N(1,b+1;1)}{\mathcal{S}_N(a+1,b+1;1)}.
\end{equation}
The SAGE function \texttt{Jac(a,b,n)} in lines 16--22 evaluates the Selberg
integral $\mathcal{S}_N(a+1,b+1;1)$ for \emph{integer} values of~$N=\mathtt{n}$
using formula~\eqref{Selberg}; the value of $C$ is then computed and stored in
\texttt{leadcoef} (line~24).

This particular implementation naturally depends on $N=\mathtt{n}$ being a
positive integer. However, it is possible to evaluate \texttt{Jac(a,b,n)} in
closed form using Barnes'  $G$-function. Unfortunately, the $G$-function is
not yet implemented in SAGE; however, given any (future) implementation thereof,
the following code
\begin{lstlisting}[language=Python,stringstyle=\ttfamily,frame=single]
# *NOT* VALID SAGE CODE UNTIL BARNES' G IS IMPLEMENTED!!!
def Jac(a,b,n):
 return G(n+a+b+2.)/G(2*n+a+b+1.) \
        *G(n+a+1.)/G(a+2.) \
        *G(n+b+1.)/G(b+2.) \
        *G(n+2.)
\end{lstlisting}
would correctly compute \texttt{leadcoef} and the program would be capable of
handling general rational values of~\texttt{n}. Alternatively, the value of
\texttt{leadcoef} can be computed by any other means and manually input into the
SAGE code as a hard constant.

We now rewrite~\eqref{dHfirstcomponent} in the form
\begin{equation}
  \label{eq:tENp}
 \widetilde{E}_N^{(a,b)\prime}(t)\ =\ \left(\frac{h^{(1)}(t) - e_2'[\bb]}{t-1} - \frac{N(N+b)}{t(t-1)}\right) \tEN(t),
\end{equation}
which, together with~\eqref{eq:ENF} allows computing $\tEN(t)$ and its
derivative $\widetilde{E}_N^{(a,b)\prime}(t)$ as done in lines 58--64 of the
code. Finally, the cumulative distribution function
\begin{equation*}
1-\EN(\phi) =  1-\tEN\left(\frac{1+\cos\phi}2\right)
\end{equation*}
and its derivative (cf.~\eqref{distributionoffirsteigenphases}
and~\eqref{eq:tENp})
\begin{equation*}
  \nuN(\phi)   =  -\frac{d}{d\phi} \EN(\phi)
  = \frac{\sin\phi}2 \widetilde{E}_N^{(a,b)\prime} \left( \frac{1+\cos\phi}{2} \right).
\end{equation*}
can be computed directly, as done in lines 66--72.  (Note that
$\frac12\sin\phi=\sqrt{t(1-t)}$ if $t=\frac{1+\cos\phi}2$.)
The entire SAGE code can be obtained from the authors and is also available for
download at \url{http://www.maths.bris.ac.uk/~mancs/publications.html}.

%%%%%%%%%%%%%%%%%%%%%%%%%%%%%%%%%%%%%%%%%%%%%%%%%%%%%%%%%%%%%%%%%%%%%%%%%%%%%%%%%%%%%%%%%%%%%%%%%%%%%%%%%%%%%%%%
%%%%%%%%%%%%%%%%%%%%%%%%%%%%%%%%%%%%%%%%%%%%%%%%%%%%%%%%%%%%%%%%%%%%%%%%%%%%%%%%%%%%%%%%%%%%%%%%%%%%%%%%%%%%%%%%
%%%%%%%%%%%%%%%%%%%%%%%%%%%%%%%%%%%%%%%%%%%%%%%%%%%%%%%%%%%%%%%%%%%%%%%%%%%%%%%%%%%%%%%%%%%%%%%%%%%%%%%%%%%%%%%%

\section{Comparison of the two methods}
\label{sec:numer-painl-solv} As might be expected, our numerical
implementation in MATLAB of the Painlev\'e solver does not work
equally well in all parameter regimes. In this section we describe
the tests we have carried out on the code and the conclusions
about the parameter regimes where a robust solution can be
obtained. The MATLAB (Runge-Kutta) solver starts from initial
conditions near $\phi=0$ (that is, $t=1$) and numerically extends
the computed solution towards $\phi=\pi$ (or $\theta =
\tfrac{N}{\pi}\phi=N$ in scaled units). The power series, on the
other hand, is an expansion around $\phi=\pi$ (or, equivalently,
$t=0$); it is therefore accurate at the opposite end of the
interval on which we are solving. Hence, if the tail of our
numerical solution matches the initial behaviour of the series
solution, we are confident that the numerical solver has worked
correctly.

There are three parameters to vary in the input to the numerical solver:
$a=\alpha-1/2$, $b=\beta-1/2$ and $N$. There are also three variables we can
adjust in the MATLAB code to try to coax a solution: $t_0$, the starting point
near $t=1$ ($\phi=0$) for the numerical solver; \verb+reltol+ and \verb+abstol+,
which control the accuracy of the numerical solution.

After testing the code for various values of $N$ (integer and
non-integer) from about 1 up to 100, it appears that a good
solution can be found on a standard desktop machine in a few
seconds with $t_0=1-10^{-7}$, \verb+reltol+$=10^{-5}$ and
\verb+abstol+$=10^{-6}$ for any $-0.5\leq a\leq 0$ and $-0.5\leq
b\leq 0.5$ (the range for $b$ that is relevant to the classical
groups). In Figure \ref{goodnumerics} we see examples of code that
runs efficiently and matches the series expansion in the tail of
the distribution.

\begin{figure}[htbp]
%\vspace*{60mm}
%\hspace*{68mm}
\includegraphics[scale=.4]{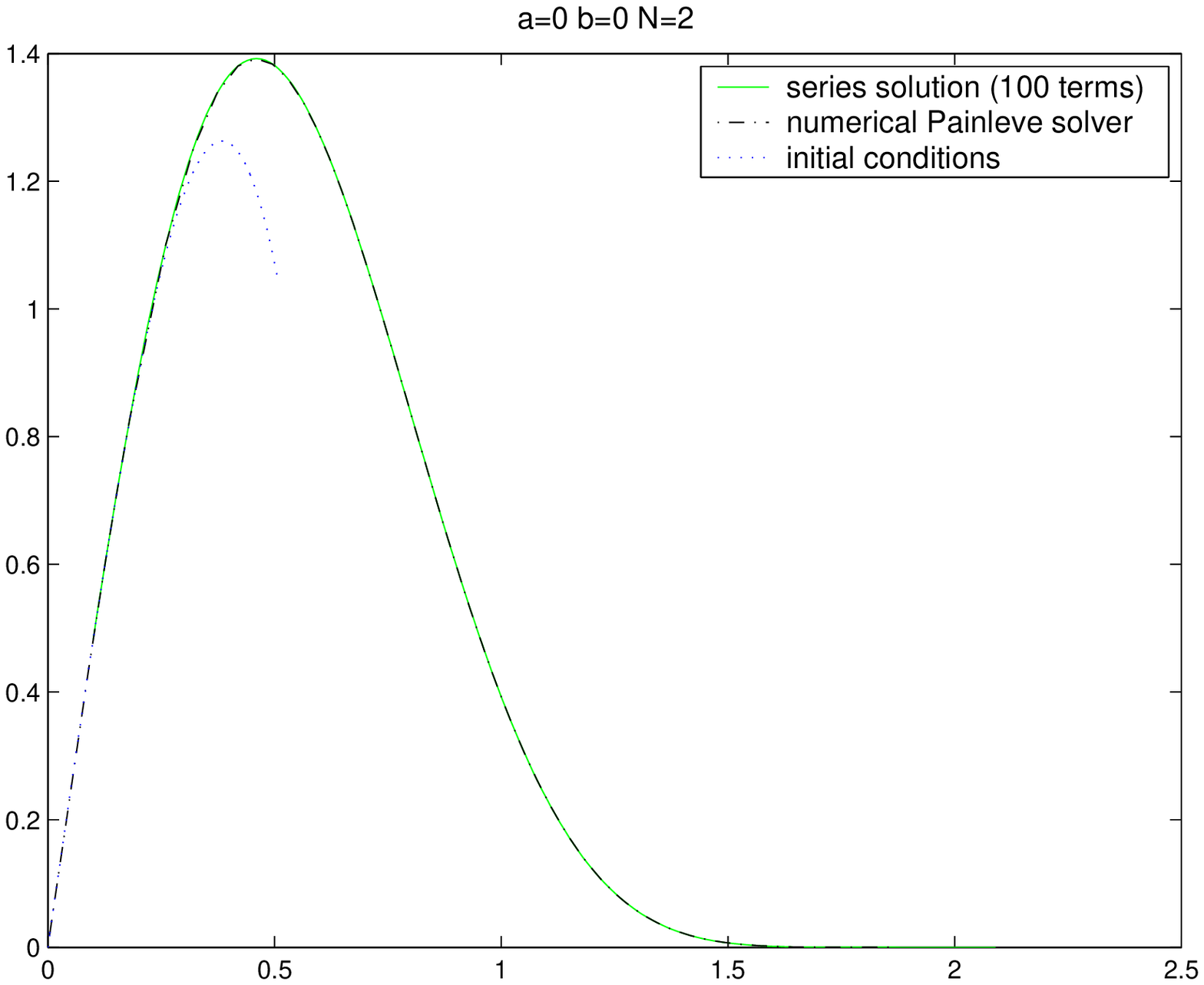} \hspace{0.1 in}
\includegraphics[scale=0.4]
{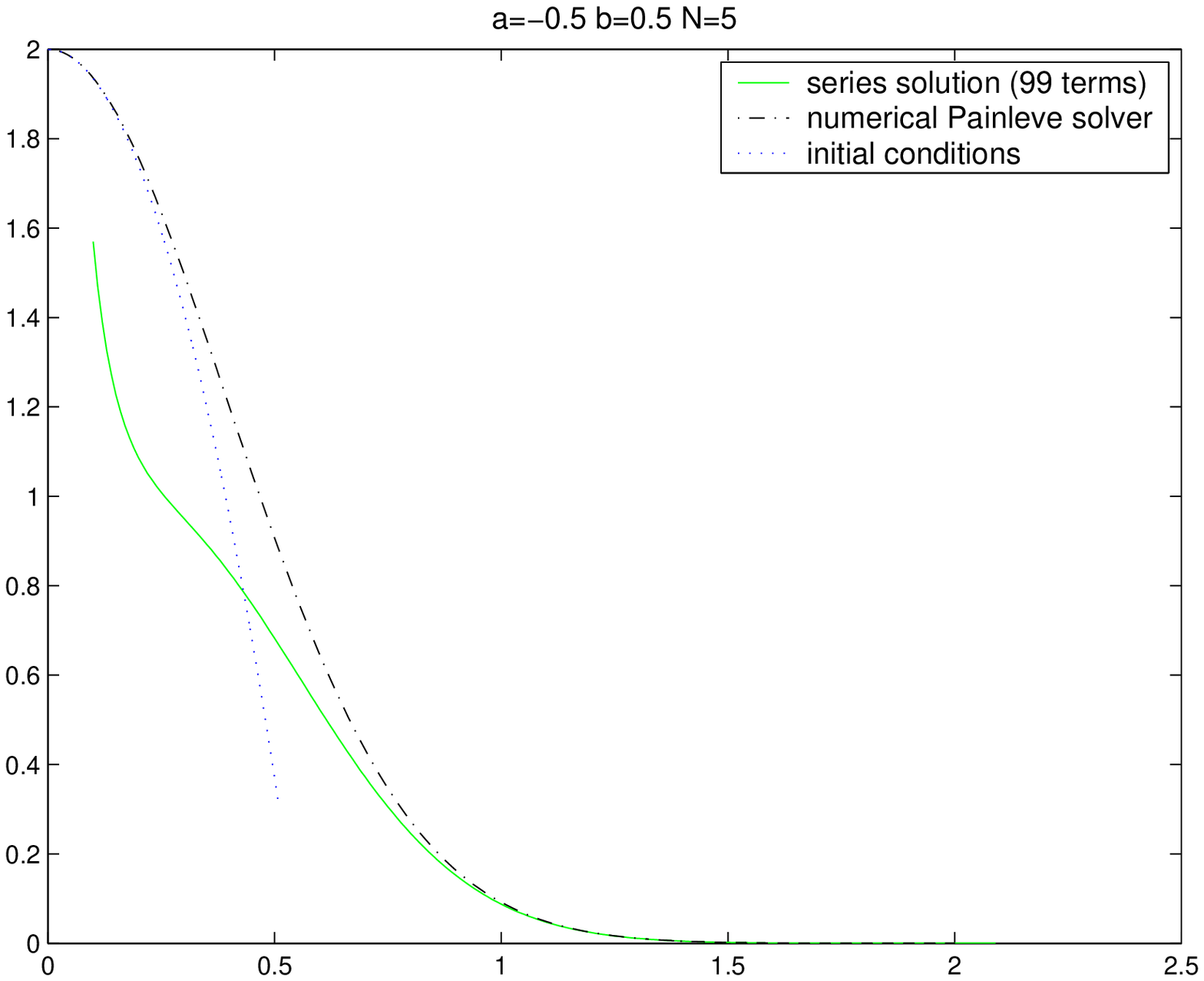}
%\vspace*{20mm}
\caption{ Plot of $\nu_N^{a+1/2,b+1/2}(\tfrac{\pi}{N}\theta)$, the
scaled distribution of the first eigenvalue.  On the left $N=2$,
$a=0$ and $b=0$.  The numerical solver was given the initial
conditions (blue dotted line) and produced the solution shown with
the dot-dashed line.  This is indistinguishable from the series
expansion (green solid line) using 100 terms. On the right $N=5$,
$a=-0.5$ and $b=0.5$. The numerical solver was given the initial
conditions (blue dotted line) and produced the solution shown with
the dot-dashed line. The tail agrees with the series expansion
(green solid line) using 99 terms. In both figures the numerical
solver was run with values of $t_0=1-10^{-7}$, $
\texttt{reltol}=10^{-5}$ and $\texttt{abstol}=10^{-6}$. }
 \label{goodnumerics}
\end{figure}

For values of $a>0$ the MATLAB solver breaks down. Trials with $a=0.001$ still
work, but already $a=0.01$ fails to produce a good solution. Moving $t_0$ away
from 1, for example to $1-10^{-2}$, helps achieve a better curve, but as can be
seen from Figure 2, the initial conditions are not close enough to the true
curve at this point to produce a valid solution. Decreasing \verb+reltol+ and
\verb+abstol+, even by a factor of 1000, does not make a visible difference to
the curve. We note that unfortunately this means that while the MATLAB solver
works very well for the group SO$(2N)$, we cannot use it to produce solutions
for SO$(2N+1)$ and USp$(2N)$.

\begin{figure}[htbp]
%\vspace*{60mm}
%\hspace*{68mm}
\includegraphics[scale=.4]{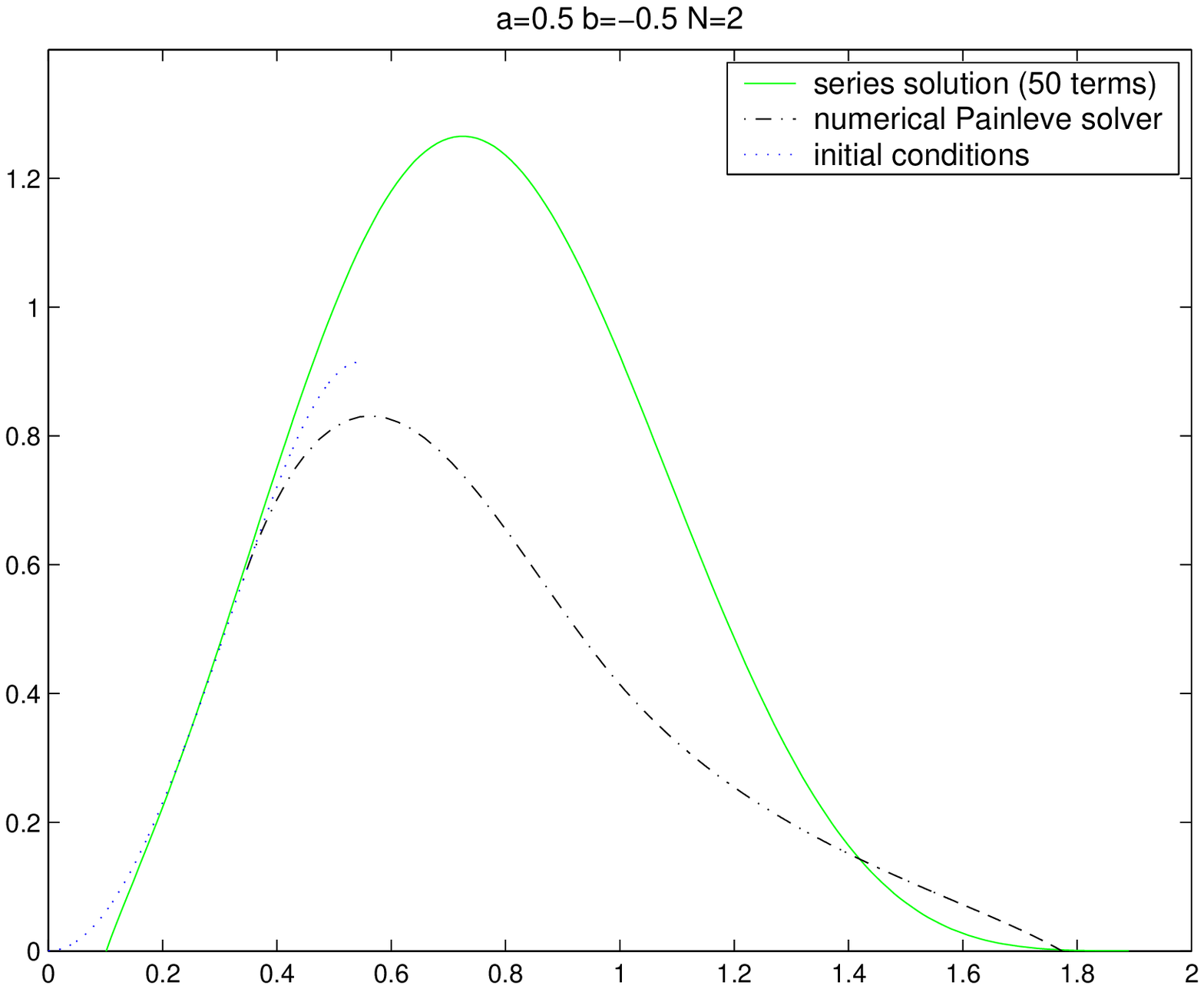} \hspace{0.1 in}
\includegraphics[scale=0.4]
{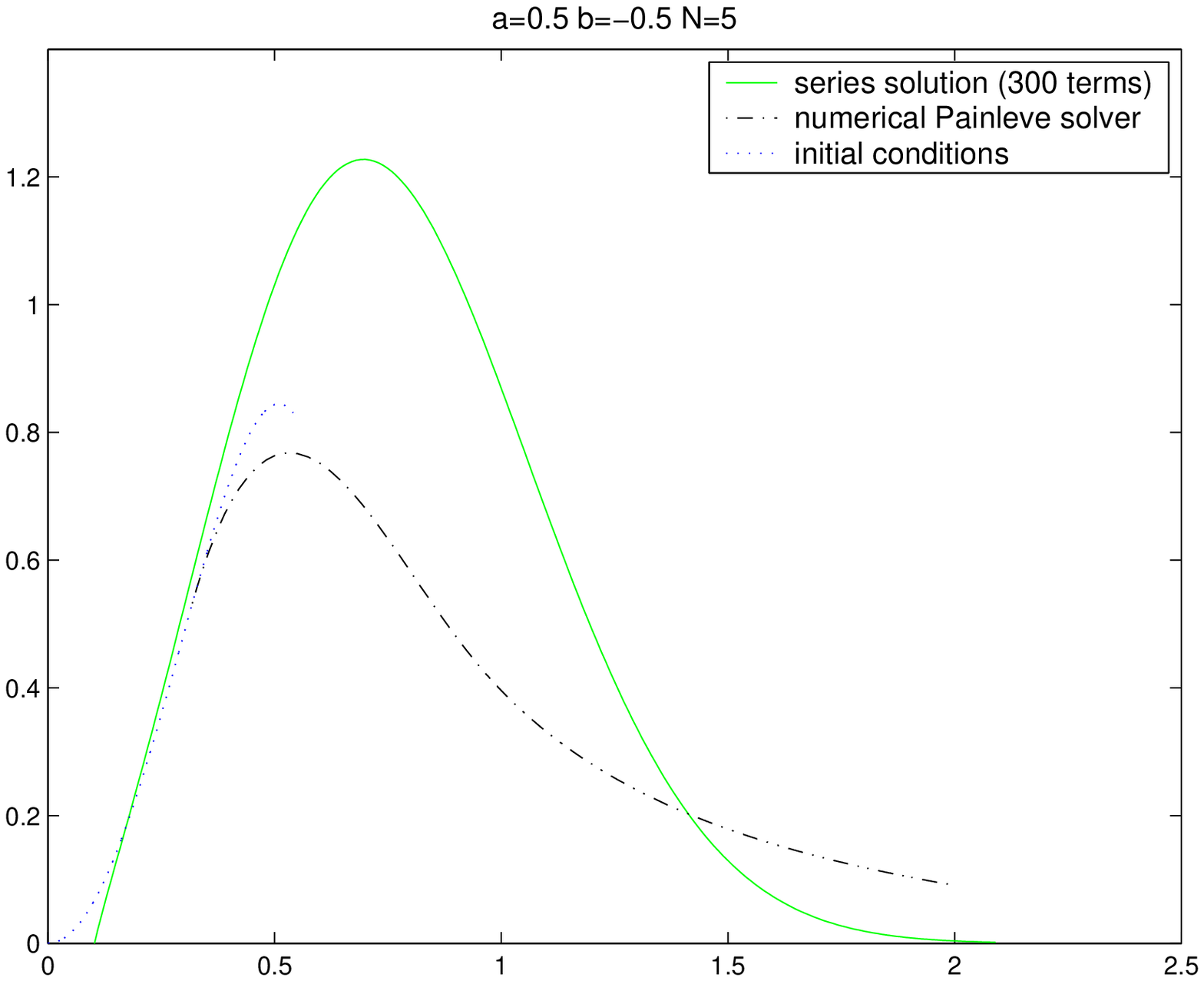}
%\vspace*{20mm}
\caption{Plot of $\nu_N^{a+1/2,b+1/2}(\tfrac{\pi}{N}\theta)$, the
scaled distribution of the first eigenvalue. On the left $N=2$,
$a=0.5$ and $b=-0.5$.  The numerical solver was given the initial
conditions (blue dotted line) and produced the solution shown with
the dot-dashed line. This fails to produce an accurate solution,
as shown by comparison with the series expansion (green solid
line) using 50 terms. On the right $N=5$, $a=0.5$ and $b=-0.5$.
The numerical solver was given the initial conditions (blue dotted
line) and produced the solution shown with the dot-dashed line.
This fails to produce an accurate solution, as shown by comparison
with the series expansion (green solid line) using 300 terms. The
value of $t_0$ in the two plots are $t_0=1-7\times10^{-2}$ and
$t_0=1-10^{-2}$ respectively, and for both plots $
\texttt{reltol}=10^{-7}$ and $\texttt{abstol}=10^{-7}$. }
 \label{badnumerics}
\end{figure}

For $a>0$ and small values of $N$ a solution valid over the whole
interval can still be glued together by matching the series
solution (with a sufficient number of coefficients) for the tail
and bulk of the curve with the known asymptotic behavior near
$\phi=0$. The right-hand plot in Figure \ref{badnumerics} shows
that this is certainly possible for $N=5$.

The series solution produces a very accurate curve with only 50 terms when
$N=2$, but the number of terms needed to obtain a solution which is meaningful
over a large interval increases with $N$, which is to be expected because the
most interesting behavior occurs near $\phi=0$ and can only be captured at the
price of using many terms in an expansion about $\phi=\pi$. A good curve for
$N=5$ requires around 300 terms. We did not produce results for higher $N$ as
the run time was prohibitive, but on a fast computer more terms could be
computed and good solutions for larger $N$ could be achieved by this method.

%%%%%%%%%%%%%%%%%%%%%%%%%%%%%%%%%%%%%%%%%%%%%%%%%%%%%%%%%%%%%%%%%%%%%%%%%%%%%%%%%%%%%%%%%%%%%%%%%%%%%%%%%%%%%%%%
%%%%%%%%%%%%%%%%%%%%%%%%%%%%%%%%%%%%%%%%%%%%%%%%%%%%%%%%%%%%%%%%%%%%%%%%%%%%%%%%%%%%%%%%%%%%%%%%%%%%%%%%%%%%%%%%
%%%%%%%%%%%%%%%%%%%%%%%%%%%%%%%%%%%%%%%%%%%%%%%%%%%%%%%%%%%%%%%%%%%%%%%%%%%%%%%%%%%%%%%%%%%%%%%%%%%%%%%%%%%%%%%%

\section{Summary}

In summary, we find that our MATLAB code for numerically solving the nonlinear
second-order differential equation (Painlev\'e~VI) of the auxiliary Hamiltonian
$h(t)$ associated to the $\tau$-function $\tEN(t)$, giving the distribution of
the first level in a Jacobi ensemble $J_N^{{(a,b)}}$, appears to work fine for
arbitrary $N$ provided the parameters lie in the ranges: $-0.5 \leq a\leq 0$ and
$-0.5\leq b \leq 0.5$. We restricted our tests to the interval $[-0.5,0.5]$ as
this is the range of interest interpolating between the classical compact groups
SO$(2N)$, SO$(2N+1)$ and USp$(2N)$. The program's numerical accuracy is
confirmed by comparing it with a power-series expansion of the solution found
using SAGE. Unfortunately the restriction on $a$ to be non-positive means that
the numerical solver cannot cope with the symplectic group USp$(2N)$ nor the odd
orthogonal group SO$(2N+1)$. For positive $a$ and small values of $N$ this
limitation can be overcome by using the series expansion to obtain the tail and
the bulk of the distribution, matching the series result with the initial
conditions for the behaviour near the origin. In a forthcoming paper
%\cite{DHKMS1}
our algorithms are applied to study the distribution of the first
zero of $L$-functions with even functional equation associated with quadratic
twists of a fixed elliptic curve. In the limit of large conductors in families
of such $L$-functions the zero statistics are expected to be modelled by
eigenvalues from SO$(2N)$.

% Likewise the family
% $\FDelta$ of quadratic twists of the Ramanujan tau $L$-function in chapter 5 has conjecturally
% orthogonal symmetry, and we can apply our program without any adjustment.
% In chapter 5 we also encounter families $\FD$ of real quadratic Dirichlet $L$-functions.
% These have conjecturally symplectic symmetry. For these regimes we need to do some adjustment
% as discussed above to get our program to work in a reasonable time. However, we leave them
% out as in this thesis we do not need the distribution of the first eigenphase for the group $USp(2N)$.
% Hence, our program can be applied to any of the groups $SO(2N), SO(2N+1)$ and $USp(2N)$
% and it then gives the distribution of the first eigenphase of the associated group.

\appendix
\section{MATLAB Code}
Here we give the MATLAB code for numerically solving the Painlev\'e VI equation associated
to the distribution of the first eigenphase of random matrix ensembles in the Jacobi ensemble
$J_N = J_N^{(\alpha,\beta)}$.

The main program is \verb+painleve6.m+ and it computes the numerical solution to the system of differential
equations~(\ref{Painlevesystem}) for the Jacobi ensemble $\JN$. For fixed variables \verb+a, b, N+ the program is
called by entering
\begin{equation} \label{mainprogram}
  \texttt{function [t,H,theta,Fp] = painleve6(a,b,N)}
\end{equation} Here \texttt{t}, \texttt{H} correspond to the variables $t$ and $H$ (see \reff{vectorH}), and
\texttt{theta} corresponds to the rescaled angular variable $\theta =
N\phi/\pi = {\frac N\pi}\cos^{-1}(2t-1)$. The output variable
\verb+Fp+ in \reff{mainprogram} is a vector of values of the rescaled
distribution
\begin{equation} \label{definitionofFp}
  \nu_N^{(a+1/2,b+1/2)}\bigg(\frac{\pi \theta}{N} \bigg) = \frac{\pi}{2 N}
  \sin\bigg(\frac{\pi \theta}{N} \bigg) \frac{\tEN(t)}{t(t-1)}
  \left[h(t) - t e_2'[\bb] + \tfrac{1}{2} e_2[\bb] \right]
\end{equation}
of the first eigenphase (the rescaling achieves mean unit spacing of the $N$
eigenphases $\theta_1,\dots,\theta_N$ on $[0,N]$), as obtained
from~(\ref{dHfirstcomponent}) via the rescaled variable $\theta$ (=
{\small\verb+theta+}).
The subsequent command
\begin{equation}
\texttt{plot(theta, Fp)}
\end{equation}
plots the distribution \verb+Fp+ (as defined in
\reff{definitionofFp}) of the first rescaled eigenphase
$\theta_{\min}$ for~$J_N$.  Notice that \verb+a=-0.5+ and
\verb+b=-0.5+ corresponds to SO$(2N)$. Likewise setting
\verb+a=0.5+ and \verb+b=-0.5+ gives SO($2N+1$) and finally
USp$(2N)$ would correspond to choosing \verb+a=0.5+ and
\verb+b=0.5+.

\vspace*{3mm}
\hrule
\vspace*{3mm}

The code of \verb+painleve6.m+ is given as follows

\begin{verbatim}
function [t,H,theta,Fp] = painleve6(a,b,N)

t0 = 1 - 1e-7;
phi0 = acos(2*t0-1);

% in parameter regions where the numerical solver is robust t0 can be
% taken to be about 1-1e-7

  b1 = (a+b)/2+N;
  b2 = (a-b)/2;
  b3 = -(a+b)/2;
  b4 = -(a+b)/2-N;
  e2p = b1*b3 + b1*b4 + b3*b4;
  e2 = e2p + b2*(b1+b3+b4);
  r = a+0.5; % note that r and s correspond to alpha and beta
  s = b+0.5;

  % initial conditions below are from Section 5. The files Hone.m
  % and Htwo.m calculate often-used ratios of gamma functions.  EN,
  % ENderiv1 and ENderiv2.m calculate E_N^{(a,b)}(phi), and its
  % first and second derivatives with respect to phi
  % (*not* with respect to t).

  root0 = sqrt(t0*(1-t0));
  E0 = EN(r,s,N,phi0);
  E0d = ENderiv1(r,s,N,phi0);
  E0dd = ENderiv2(r,s,N,phi0);
  H0 = [ E0; ...
         t0*e2p - 0.5*e2 + root0*E0d/E0; ...
         e2p + (0.5-t0)/root0 * E0d/E0 - E0dd/E0 + (E0d/E0)^2 ];

  % First component of H will be   \tilde{E}_N^{(a,b)}(t)
  % Second component of H will be the auxiliary hamiltonion h(t)
  % Third component of H will be h'(t)
  % H0 contains the initial conditions at t=t0 (Note: t0 is a
  % number very close to 1, not close to zero!!

  opts=odeset('reltol',1e-5,'abstol',1e-6);

  % a command like "opts=odeset('reltol',1e-5,'abstol',1e-6);" works
  % fine for parameter ranges where the numerical solver is robust,
  % and the program just takes a few seconds/minutes to run.
  % As a becomes positive the differential equation
  % solver has trouble and we don't get a correct solution
  % when N is large, a plot on a better scale is produced by
  % replacing the second argument in the range [t0,0.01] with
  % 0.5*cos(2*pi/N)+1/2

  [t,H] = ode45(@p6diff,[t0,0.01],H0,opts,b1,b2,b3,b4,e2p,e2);

  F = H(:,1);
  h = H(:,2);
  % theta is the scaled angular variable: theta=(N/pi)*acos(2*t-1)
  theta = (N/pi)*acos(2*t-1);

  % Fp(theta) is the distribution of the first eigenvalue in scaled
  % variables
  Fp = (pi/2/N)*sin(pi*theta/N).*F./t./(t-1).*(h-e2p*t+e2/2);
\end{verbatim}
\vspace*{3mm}
\hrule
\vspace*{3mm}
The system \reff{F(H)} is defined
in the file \verb+p6diff.m+. The code is given as
\begin{verbatim}
function dH = p6diff(t,H,b1,b2,b3,b4,e2p,e2)
dH = zeros(3,1);
dH(1) = H(1)/t/(t-1)*(H(2)-e2p*t+e2/2);
dH(2) = H(3);
dH(3) = +sqrt( ...
              ( ...
               (H(3)+b1^2)*(H(3)+b2^2)*(H(3)+b3^2)*(H(3)+b4^2) ...
               - ( H(3)*(2*H(2)-(2*t-1)*H(3)) + b1*b2*b3*b4 )^2 ...
              ) / H(3) ...
             ) /t/(1-t);
\end{verbatim}

\vspace*{3mm}
\hrule
\vspace*{3mm}

The often used ratios of gamma functions $H_1$ and $H_2$ (see \reff{HoneforMatlab} and \reff{HtwoforMatlab}) are coded
in the files \verb+Hone.m+ and \verb+Htwo.m+ given as
\begin{verbatim}
function H1 = Hone(r,s,N)
H1 = gamma(r + N +0.5) * gamma(r + s + N) / 2^(2*r) ...
     / gamma(N + 1) / gamma(r + 1.5) / gamma(r + 0.5) ...
     / gamma(s + N - 0.5);
\end{verbatim}
and
\begin{verbatim}
function H2 = Htwo(r,s,N)
H2 = gamma(r + N + 0.5) * gamma(r + s + N + 1) / 2^(2*r+ 1) ...
     / gamma(N + 1) / gamma(r + 2.5) / gamma(r + 0.5) ...
     / gamma(s + N - 0.5);
\end{verbatim}

\vspace*{3mm}
\hrule
\vspace*{3mm}

Finally, we provide the code for \verb+EN.m, ENderiv1.m+, and \verb+ENderiv2.m+

\begin{verbatim}
function E = EN(r,s,N,phi)
% This is an expansion in phi around phi=0 of E_N^{(a,b)}(phi).
\end{verbatim}
\verb+% See+ \texttt{\reff{EGNforMatlab}}.
\begin{verbatim}
exponent = 2*r+1;
E = 1;
E = E - N*Hone(r,s,N)*phi.^exponent/exponent;
exponent = exponent + 2;
E = E + N*( ...
           (N-1)*Htwo(r,s,N)*phi.^exponent ...
           + (r/12.0+s/4)*Hone(r,s,N)*phi.^exponent ...
          )/exponent;


function Ed = ENderiv1(r,s,N,phi)
% This is \frac{d}{d\phi}E_N^{(a,b)}(phi).
\end{verbatim}
\verb+% See+ \texttt{\reff{conditiontwoMatlab}}.
\begin{verbatim}
exponent = 2*r;
Ed = -N*Hone(r,s,N)*phi.^exponent;
exponent = exponent + 2;
Ed = Ed + N*( ...
              (N-1)*Htwo(r,s,N)*phi.^exponent ...
              + (r/12.0+s/4)*Hone(r,s,N)*phi.^exponent ...
            );


function Edd = ENderiv2(r,s,N,phi)
% This is \frac{d}{d\phi}E_N^{(a,b)}(phi).
\end{verbatim}
\verb+% See+ \texttt{\reff{conditionthreeMatlab}}.
\begin{verbatim}
exponent = 2*r-1;
Edd = -2*r*N*Hone(r,s,N)*phi.^exponent;
exponent = exponent+2;
Edd = Edd ...
      + N*(exponent+1)*( ...
                         (N-1)*Htwo(r,s,N)*phi.^exponent ...
                         + (r/12.0+s/4)*Hone(r,s,N)*phi.^exponent ...
                        );
\end{verbatim}

\section{SAGE Code}
\label{sec:sage-code}

Below follows the SAGE code implementing the symbolic power series
evaluation of the $\tau$-function. Note that the parameters $a,b,N$
correspond to \texttt{a, b, N}, which along with \texttt{DEGREE} (the
degree of the sought-after truncation of the power series) are
hard-coded in the first four lines. Note also that the code provided
only works for rational values of $a,b$ and integer~$N$ (cf.,
Section~\ref{sec:power-series}). Once run, the code defines functions \texttt{E(t), Ep(t),
  pcummul(phi)} and \texttt{nu(phi)} implementing
$\tEN(t),\frac{d}{dt}{\tEN}(t), 1-\EN(\phi)$ and $\nu(\phi)$,
respectively. In particular, \texttt{nu} has been used to produce the
plots in figures~1 and~2.

Note that in Python syntax any text following the literal \texttt{\#} is simply
a comment.

\lstset{language=Python, numbers=left, numberstyle=\tiny, numbersep=5pt, stringstyle=\ttfamily}
\begin{lstlisting}
DEGREE = 50   # degree (plus 1) of the truncated power series
a = -1/2
b = -1/2
n = 4

b1 = (a+b)/2+n
b2 = (a-b)/2
b3 = -(a+b)/2
b4 = -(a+b)/2-n
e2p = b1*b3 + b1*b4 + b3*b4
e2 = e2p + b2*(b1+b3+b4)

R.<x> = QQ['x']    # R = Q[x]
S.<t> = PowerSeriesRing(R,default_prec=DEGREE)
                   # S = Q[x][[t]]
def Jac(a,b,n):
  jac=1.
  for i in xrange(n):
      jac *= real(gamma(a+i+1.)) * real(gamma(b+i+1.))\
             * real(gamma(i+2.)) / real(gamma(a+b+n+i+1.))
  return jac
# NOTE: Jac (as above) only works for *integer* values of n

leadcoef = Jac(0,b,n)/Jac(a,b,n)
leadexp = n*(n+b)
g = [-e2/2-leadexp, e2p+leadexp*(1+a/(b+2*n))]
h = g[0] + t*g[1]

X = var('X')        # symbolic Maxima variable

# BEGIN MAIN LOOP: Finds successive coefficients of series h(t)
for i in xrange (2,DEGREE):
 h = h + x*t^i + O(t^(i+2))
 hp = h.derivative()
 hpp = hp.derivative()
 PEZ = hp*(hpp*t*(1-t))^2 + (hp*(2*h-hp*(2*t-1))+b1*b2*b3*b4)^2 \
       - (hp+b1^2)*(hp+b2^2)*(hp+b3^2)*(hp+b4^2)
 LHS = PEZ[i](X)   # Substitute x=X in p_i(x)
 solucion = solve (LHS==0, X)   # Find root of p_i(X)
 g.append(QQ(solucion[0].rhs()))       # The root becomes
                                       #  the next coefficient
 # Now reconstruct h up to degree i using the newly-found g[i]
 h = g[0]+g[1]*t+g[2]*t^2
 for j in xrange (3,i+1):
     h += g[j]*t^j
 # END OF MAIN LOOP

h = h + O(t^DEGREE)

h1 = (h-h[0])/t
F = h1
F -= e2p+leadexp
F /= t-1+O(t^DEGREE)
F = F.integral()
F = F.exp()
F = F.truncate(DEGREE+1)

htilde = h1 - e2p
htilde = htilde.truncate(DEGREE-1)

def E(u):
  return leadcoef * u^leadexp * F(u)
def Ep(u):
  return (htilde(u)/(u-1) + leadexp/u/(1-u)) * E(u)

def pcummul(theta):
  return 1-E((1+cos(theta))/2)
def nu(theta):
  if theta < 1.e-6 or theta > pi - 1.e-6:
      return 0
  t = (1+cos(theta))/2
  return sqrt(t*(1-t)) * Ep(t)
\end{lstlisting}

\bibliography{painleve6}{}
\bibliographystyle{plain}

% \begin{thebibliography}{ACBDEF}

% \bibitem[DHKMS]{DHKMS1}
% E. Due~nez, D.K. Huynh, J.P. Keating, S.J. Miller, and N.C. Snaith, \emph{A model for
% elliptic curve $L$-functions of finite conductor}, in preparation.

% \bibitem[EP]{Edel06}
% A. Edelman and P. Persson, \emph{Numerical Methods for Eigenvalue Distribution of
% Random Matrices}, preprint, 2005. \texttt{http://arxiv.org/abs/math-ph/0501068}

% \bibitem[FW]{ForrWitte04}
% P. J. Forrester and N. S. Witte, \emph{Application of the $\tau$-function theory of Painlev\'e
% equations to random matrices: PV I , the JUE, CyUE, cJUE and scaled limits}, Nagoya
% Math. J. \textbf{174} (2004), 29--114.

% \bibitem[Meh]{Mehta91}
% M.L. Mehta, Random Matrices, Academic Press, 2nd edition, Academic Press,
% Boston, 1991.

% \bibitem[Mil1]{Mil04a}
% S.J. Miller, \emph{One- and two-level densities for rational families of elliptic curves: evi-
% dence for the underlying group symmetries}, Compos. Math. \textbf{140} (2004), 952--992.

% \bibitem[Mil2]{Mil06}
% S.J. Miller, \emph{Investigations of zeros near the central point
% of elliptic curve $L$-functions} (with an appendix by E. Due\~nez), Experimental Mathematics \textbf{15} (2006), no. 3, 257--279. (E-print: math.NT/0508150).

% \bibitem[Y]{Young06}
% M. P. Young, \emph{Low-lying zeros of families of elliptic curves}, J. Amer. Math. Soc. \textbf{19} (2006),
% 205--250.

% \end{thebibliography}

%\ \\

\ \\

\end{document}